\documentclass[a4paper, 11pt,reqno]{amsart}
\usepackage[top = 1in, bottom = 1in, left=0.95in, right=0.95in]{geometry}

\usepackage[utf8]{inputenc}
\usepackage[activeacute,english]{babel}
\usepackage[T1]{fontenc}
\usepackage{verbatim}
\usepackage{lmodern}

\usepackage{xcolor}
\usepackage{hyperref}
\hypersetup{colorlinks, linkcolor={red!50!black}, citecolor={green!50!black}, urlcolor={blue!50!black}}

\usepackage{microtype}

\usepackage[shortlabels]{enumitem}

\usepackage[linesnumbered,ruled,vlined]{algorithm2e}

\usepackage{amsmath,amssymb,amsfonts}
\usepackage{amsthm}

\usepackage{tikz}
\usetikzlibrary{decorations.pathreplacing}

\def\NN{\mathbb{N}}
\DeclareMathOperator{\prob}{\mathbb{P}}
\DeclareMathOperator{\expec}{\mathbb{E}}

\newcommand{\whp}{whp}
\newcommand{\Whp}{Whp}
\newcommand{\cB}{\mathcal{B}}
\newcommand{\cP}{\mathcal{P}}
\newcommand{\cK}{\mathcal{K}}
\newcommand{\cF}{\mathcal{F}}
\newcommand{\cA}{\mathcal{A}}
\newcommand{\cE}{\mathcal{E}}
\newcommand{\EE}{\mathbb{E}}

\newcommand{\eps}{\varepsilon}
\newcommand{\Bi}{\mathrm{Bin}}
\newcommand{\Tstop}{T_{\mathrm{stop}}}
\newcommand{\pathfinder}{\textnormal{\texttt{Pathfinder}}}

\renewcommand{\Pr}{\mathbb{P}}

\theoremstyle{theorem}
\newtheorem{theorem}{\textbf{Theorem}}
\newtheorem{corollary}[theorem]{\textbf{Corollary}}
\newtheorem{lemma}[theorem]{\textbf{Lemma}}
\newtheorem{claim}[theorem]{\textbf{Claim}}
\newtheorem{question-formatless}{\textbf{Question}}
\newtheorem{proposition}[theorem]{\textbf{Proposition}}
\theoremstyle{definition}
\newtheorem{definition}[theorem]{\textbf{Definition}}

\theoremstyle{remark}
\newtheorem{remark}[theorem]{\textbf{Remark}}

\begin{document}

\title[Longest paths in random hypergraphs]{Longest paths in random hypergraphs}

\author[O. Cooley, F. Garbe, E. K. Hng, M. Kang, N. Sanhueza-Matamala, J. Zalla]{Oliver Cooley$^{*}$,
Frederik Garbe$^{\dagger}$,
Eng Keat Hng$^{\ddagger}$,
Mihyun Kang$^{*}$,
Nicol\'as Sanhueza-Matamala$^{\S}$,
Julian Zalla$^{*}$}

\renewcommand{\thefootnote}{\fnsymbol{footnote}}

\footnotetext[1]{Supported by Austrian Science Fund (FWF): I3747, W1230, \texttt{\{cooley,kang,zalla\}@math.tugraz.at}, Institute of Discrete Mathematics, Graz University of Technology, Steyrergasse 30, 8010 Graz, Austria.}

\footnotetext[2]{Supported by GA\v{C}R project 18-01472Y and RVO:
67985840, \texttt{garbe@fi.muni.cz}, Masaryk University, Faculty of
Informatics, Botanick\'a 68A, 602 00, Brno, Czech Republic.}

\footnotetext[3]{Supported by an LSE PhD Studentship, \texttt{e.hng@lse.ac.uk}, Department of Mathematics, London School of Economics, Houghton Street, London, WC2A 2AE, United Kingdom.}

\footnotetext[4]{Supported by the Czech Science Foundation, grant number GA19-08740S
   with institutional support RVO: 6798580. \texttt{nicolas@sanhueza.net}, The Czech Academy of Sciences, Institute of Computer Science, Pod Vod\'{a}renskou v\v{e}\v{z}\'{\i} 2, 182 07 Prague, Czech Republic.}

\renewcommand{\thefootnote}{\arabic{footnote}}

\begin{abstract}
Given integers $k,j$ with $1\le j \le k-1$, we consider the length of the
longest $j$-tight path in the binomial random $k$-uniform hypergraph $H^k(n,p)$.
We show that this length undergoes a phase transition from logarithmic length
to linear and determine the critical threshold, as well as proving upper and lower
bounds on the length in the subcritical and supercritical ranges.

In particular, for the supercritical case we introduce the
\pathfinder\ algorithm, a depth-first search algorithm which discovers
$j$-tight paths in a $k$-uniform hypergraph. We prove that, in the supercritical
case, with high probability this algorithm will find a long $j$-tight path.
\end{abstract}

\maketitle

\section{Introduction} \label{sec:intro}

The celebrated phase transition result of Erd\H{o}s and R\'enyi~\cite{ErdosRenyi60} for random
graphs states, in modern terminology, that the binomial random graph\footnote{$G(n,p)$
is the random graph on vertex set $[n]:= \{1,\ldots,n\}$, in which each pair of
vertices is connected by an edge with probability $p$ independently.}
$G(n,p)$ 
displays a dramatic change in the order of the largest component when
$p$ is approximately $1/n$.
If $p$ is slightly smaller than $1/n$, then whp\footnote{Short for ``with high probability'',
meaning with probability tending to $1$ as $n$ tends to infinity.}
all components
are at most of logarithmic order, while if $p$ is slightly larger than $1/n$,
then there is a unique ``giant'' component of linear order and all other components are again of logarithmic order.

\subsection{Paths in random graphs} 
While by definition any two vertices in a component are connected by a path,
there is not necessarily a correlation between the order of the component
and the lengths of such paths. Of course, if a component is small, then it
can only contain short paths, but if a component is large, this does not guarantee
the existence of a long path. Nevertheless, Ajtai, Koml\'os and Szemer\'edi~\cite{AjtaiKomlosSzemeredi81}
showed that if $p$ is larger than $1/n$, then whp $G(n,p)$ does indeed
contain a path of linear length.

Incorporating various extensions of the results of Erd\H{o}s and R\'enyi and of
Ajtai, Koml\'os and Szemer\'edi
by Pittel~\cite{Pittel88}, by {\L}uczak~\cite{Luczak91}, and by Kemkes and Wormald~\cite{KemkesWormald13},
gives the following.

\begin{theorem}\label{thm:graphpathphasetransitionconstanteps}
Let $L$ denote the length of the longest path in $G(n,p)$.
\begin{enumerate}[label=\textnormal{(\roman*)}]
\item If $0<\eps< 1$ is a constant and $p=\frac{1-\eps}{n}$, then
for any $\omega=\omega(n)$ such that $\omega\xrightarrow{n\to \infty} \infty$,
\whp
$$
\frac{\ln n-\omega}{-\ln(1-\eps)} \le L \le \frac{\ln n+\omega}{-\ln(1-\eps)}.
$$
\item If $0<\eps = \eps(n) = o(1)$ satisfies $\eps^5 n \to \infty$ and $p=\frac{1+\eps}{n}$, then \whp
$$
\left(\frac{4}{3}+o(1)\right)\eps^2 n \le L \le (1.7395+o(1))\eps^2 n.
$$
\end{enumerate}
\end{theorem}

Let us also note that very recently, Anastos and Frieze~\cite{AnastosFrieze19}
determined $L$ asymptotically in the range when $p=c/n$ for a sufficiently
large constant $c$ (in particular, $c$ is much larger than $1$).

In fact, the bounds in the supercritical case followed from results about the length of the longest \emph{cycle}.
These original results also hold
under the weaker assumption that $\eps^3 n\to \infty$,
and in particular the lower bound for paths is still valid even with this weaker assumption.
For the upper bound, however,
the standard sprinkling argument to show that
the longest cycle is not significantly shorter than the longest path
breaks down when $\eps = O(n^{-5})$,
and so we would no longer obtain the upper bound on $L$ in the
supercritical case.

In this paper we generalise Theorem~\ref{thm:graphpathphasetransitionconstanteps}
for various notions of paths
in random \emph{hypergraphs}.

\vspace{-0.1cm}

\subsection{Main result: paths in hypergraphs}

Given a natural number $k$,
a \emph{$k$-uniform hypergraph}
consists of a vertex set $V$ and an edge set
$E$, where each edge consists of precisely $k$ distinct vertices.
Thus a $2$-uniform hypergraph is simply a graph.
Let $H^k(n,p)$ denote the binomial random $k$-uniform hypergraph
on vertex set $[n]$ in which each set of $k$ distinct vertices forms an edge with probability $p$ independently. Thus in particular
$H^2(n,p)=G(n,p)$.

There are several different ways of generalising the concept of paths in $k$-uniform hypergraphs.
One important concept leads to a whole family of different types of paths
which have been extensively studied. Each path type is defined by a parameter
$j\in [k-1]$, which is a measure of how tightly connected the path is.
Formally, we have the following definition.
\begin{definition}
Let $k,j\in\NN$ satisfy $1 \leq j \leq k-1$ and let $\ell \in \mathbb{N}$.
A \emph{$j$-tight path of length $\ell$} in a $k$-uniform hypergraph consists of
a sequence of distinct vertices $v_1,\ldots,v_{\ell(k-j)+j}$ and a sequence of edges
$e_1,\ldots,e_\ell$, where
$e_i= \{ v_{(i-1)(k-j)+1},\ldots, v_{(i-1)(k-j)+k} \}$ for $i=1,\ldots,\ell$, see Figure~\ref{fig:hyperpath}.
\end{definition}

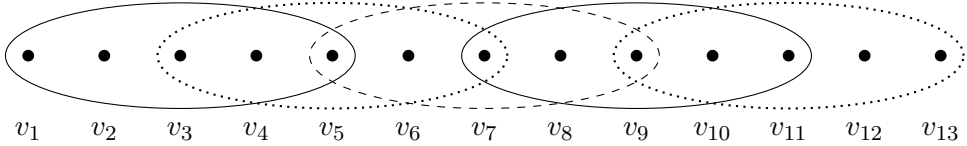
\begin{figure}[h]
\begin{center}
\begin{tikzpicture}
\foreach \x in {1,2,3,4,5,6,7,8,9,10,11,12,13}
{
\coordinate (\x) at (\x-1,0);
\filldraw (\x) circle (2pt);
\node (label\x) at (\x-1,-1) {$v_{\x}$};
}
\draw (2,0) ellipse (2.3cm and 0.7cm);
\draw[thick,dotted] (4,0) ellipse (2.3cm and 0.7cm);
\draw[dashed] (6,0) ellipse (2.3cm and 0.7cm);
\draw (8,0) ellipse (2.3cm and 0.7cm);
\draw[thick,dotted] (10,0) ellipse (2.3cm and 0.7cm);
\end{tikzpicture}
\end{center}
\caption{A $3$-tight path of length $5$ in a $5$-uniform hypergraph}\label{fig:hyperpath}
\end{figure}

Note that the case $k=2$ and $j=1$ simply defines a path in a graph. 
For $k\geq 3$, the case $j=1$ is often called a \emph{loose path}, while the 
case $j=k-1$ is often called a \emph{tight path}.

The main result of this paper is a phase transition result for $j$-tight paths
similar to Theorem~\ref{thm:graphpathphasetransitionconstanteps}.

\vspace{-0.1cm}

\begin{definition}
We use the notation
$
f\ll g
$
to mean that $f\le g/C$ for some sufficiently large constant $C$,
and similarly $f\gg g$ to mean that $f\ge Cg$ for some sufficiently large constant $C$.
\end{definition}

\begin{theorem}\label{thm:mainresult}
	Let $k,j\in\NN$ satisfy $1 \leq j \leq k-1$. Let $a\in\NN$ be the unique integer satisfying
	$1 \leq a \leq k-j$ and $a \equiv k \bmod (k-j)$.
	Let $\eps = \eps(n) \ll 1$ satisfy $\eps^3 n \xrightarrow{n\to\infty} \infty$
	and let
	$$
	p_0 = p_0(n;k,j) := \frac{1}{\binom{k-j}{a} \binom{n-j}{k-j}}.
	$$
	Let $L$ be the length of the longest $j$-tight path in $H^k(n,p)$.
	
	\begin{enumerate}[label=\normalfont{(\roman*)}]
		\item \label{item:thm-subcritical} If $p = (1-\eps)p_0$, then \whp
		\[ \frac{j \ln n- \omega + 3 \ln \eps}{- \ln (1 - \eps)} \leq L \leq \frac{j \ln n+ \omega}{-\ln(1-\eps)} ,\]
		for any $\omega=\omega(n)$ such that $\omega\xrightarrow{n\to\infty} \infty$.
		\item \label{item:thm-supercritical} If $p = (1+\eps)p_0$
		and $j\ge 2$, then for any $\delta$ satisfying $\delta \gg \max{\{\eps,\frac{\ln n}{\eps^2 n}\}}$, 
		\whp
		\[ (1 - \delta)\frac{\eps n}{(k-j)^2} \leq L \leq (1 + \delta)\frac{2 \eps n}{(k-j)^2}.\]
		\item \label{item:thm-loose} If $p=(1+\eps)p_0$ and $j = 1$,
		then for all $\delta \gg \eps$ satisfying $\delta^2\eps^3 n \xrightarrow{n\to \infty} \infty$, \whp
		 \[ (1 - \delta)\frac{\eps^2 n}{4(k-1)^2} \leq L \leq (1 + \delta)\frac{2 \eps n}{(k-1)^2}.\]
	\end{enumerate}
\end{theorem}

In other words, we have a phase transition at threshold $p_0$.

We will prove the upper bounds in all three cases using the first moment method.
The lower bound in the subcritical case, i.e.\ in~\ref{item:thm-subcritical},
will be  proved using the second moment method---while the strategy
is standard, there are significant technical complications to be overcome.
However, the second moment method is not strong enough in the \emph{supercritical} cases,
and therefore we will prove the lower bounds in~\ref{item:thm-supercritical}
and~\ref{item:thm-loose} by introducing the \pathfinder\ search algorithm
which explores $j$-tight paths in $k$-uniform hypergraphs, and which is the
main contribution of this paper. The algorithm is based on a depth-first search process,
but it is a rather delicate task to design it in such a way that it both
correctly constructs
$j$-tight paths and also admits reasonable probabilistic analysis. We will analyse
the likely evolution of this algorithm and prove that whp
it discovers a $j$-tight path of the appropriate length.

To help interpret Theorem~\ref{thm:mainresult},
let us first observe that the results become stronger for smaller $\delta$,
so $\delta$ may be thought of as an error term.
Furthermore, in all cases of the theorem we may choose $\delta$ to be no
larger than an arbitrarily small constant, while in some cases we may even
have $\delta \to 0$.
In the subcritical regime (Theorem~\ref{thm:mainresult}\ref{item:thm-subcritical}), note that
$-\ln(1-\eps) = \eps + O(\eps^2)$ and that the
term $3\ln \eps$ in the lower bound becomes negligible (and in particular
could be incorporated into $\omega$) if $\eps$ is constant.
For smaller $\eps$, however, it represents a gap between the lower and upper bounds.
In the supercritical case for $j\ge 2$ (Theorem~\ref{thm:mainresult}\ref{item:thm-supercritical}), the length $L$ is certainly of order
$\Theta(\eps n)$, but the lower and upper bounds differ by approximately a multiplicative factor
of $2$. In the supercritical case for $j=1$ (Theorem~\ref{thm:mainresult}\ref{item:thm-loose}), the lower and upper bounds
differ by a multiplicative factor of $\Theta(\eps)$.
This has subsequently been improved by Cooley, Kang and Zalla~\cite{CooleyKangZalla21},
who lowered the upper bound to within a constant of the lower bound
by analysing a structure similar to the $2$-core in random hypergraphs.
We will discuss all of these bounds and how they might be improved
in more detail in Section~\ref{sec:concluding}.

\begin{remark}\label{rem:weaker}
In fact, the statement of Theorem~\ref{thm:mainresult} has been slightly
weakened compared to what we actually prove in order to improve the clarity.
More precisely, the full strength of the assumption on $\delta$
in~\ref{item:thm-loose} is only
required for the lower bound; the upper bound would in fact hold for
any $\delta \gg \max\{\eps,\frac{\ln n}{\eps^2 n}\}$ as in~\ref{item:thm-supercritical}
(c.f.\ Lemma~\ref{lem:upperbound}).
Furthermore, the assumption that $\delta \gg \frac{\ln n}{\eps^2 n}$ in~\ref{item:thm-supercritical}
is only needed for the upper bound; the lower bound holds with just the
assumption that $\delta \gg \eps$ (c.f.\ Lemma~\ref{lem:mainlemmastoppingtime}).
\end{remark}

\subsection{Related work}

The study of $j$-tight paths (and the corresponding notion of $j$-tight cycles)
has been a central theme in hypergraph theory,
with many generalisations of classical graph results,
including Dirac-type and Ramsey-type 
(see~\cite{KOSurvey14,MubayiSukSurvey,ZhaoSurvey} for surveys),
as well as Erd\H{o}s-Gallai-type results~\cite{ABCM17,GyoriKatonaLemons16}.

There has also been some work on $j$-tight cycles in random hypergraphs.
Dudek and Frieze~\cite{DudekFrieze11,DudekFrieze13} determined the thresholds for the appearance of
both loose and tight Hamilton cycles in $H^k(n,p)$,
as well as determining the threshold for a $j$-tight Hamilton cycle up to
a multiplicative constant.
Recently, Narayanan and Schacht~\cite{NarayananSchacht2019} pinpointed the
precise value of the sharp threshold for the appearance of $j$-tight Hamilton cycles in $k$-uniform hypergraphs,
provided that $k > j > 1$.

Theorem~\ref{thm:mainresult} addresses a range when $p$ is significantly smaller than the threshold for a $j$-tight Hamilton cycle,
and consequently the longest $j$-tight paths are far shorter. Recently Cooley~\cite{Cooley21}
has extended the lower bound in Theorem~\ref{thm:mainresult}\ref{item:thm-supercritical}
to the range when $p=cp_0$ for some constant $c>1$, and shown that with
a much more difficult version of the common ``sprinkling'' argument,
one can also find a $j$-tight \emph{cycle} of approximately the same length.

Recall that for random graphs, the phase transition thresholds for the length of the 
longest path and the order of the largest component are both $1/n$.
It is therefore natural to wonder whether something similar holds for $j$-tight
paths in random hypergraphs, since
for each $1\le j \le k-1$, there is a notion of connectedness
that is closely related to $j$-tight paths: two $j$-tuples $J_1,J_2$ of vertices
are $j$-tuple-connected if there is a sequence of edges $e_1,\ldots,e_\ell$ such
that $J_1\subset e_1$ and $J_2\subset e_\ell$, and furthermore
any two consecutive edges $e_i,e_{i+1}$ intersect in at least $j$ vertices.
A $j$-tuple component is a maximal collection of pairwise $j$-tuple-connected
$j$-sets.

The threshold for the emergence of the giant $j$-tuple component
in $H^k(n,p)$ is known to be
$$
p_g=p_g(n;k,j)=\frac{1}{\left(\binom{k}{j}-1\right)\binom{n-j}{k-j}}.
$$
The case $k=2$ and $j=1$ is the classical graph result of Erd\H{o}s and R\'enyi.
The case $j=1$ for general $k$ was first proved by
Schmidt-Pruzan and Shamir~\cite{SchmidtShamir85}. The case of general
$k$ and $j$ was first proved by Cooley, Kang, and
Person~\cite{CooleyKangPerson18}.

One might expect the threshold for the emergence of a $j$-tight path of linear
length to have the same threshold. However, it turns out that this is only true
in the case when $j=1$. More precisely, 
in the case $j=1$, the probability threshold of $\frac{1}{(k-1)\binom{n-j}{k-j}}$
given by Theorem~\ref{thm:mainresult} matches the threshold for the emergence
of the giant (vertex-)component.
However, for $j\ge 2$, the two thresholds do not match. A heuristic explanation for this is
that when exploring a $j$-tuple component via a (breadth-first or depth-first)
search process, each time we find an edge
we may continue exploring a \emph{$j$-tuple component} from any of the $\binom{k}{j}-1$ 
new $j$-sets within this edge (all are new except the $j$-set from which
we first found the edge).
However, when exploring a \emph{$j$-tight path}, the restrictions on the structure mean that
not all $j$-sets within the edge may form the last $j$ vertices of the path. For $a$ as defined in Theorem~\ref{thm:mainresult}, it will turn out that we only have $\binom{k-j}{a}$ choices for
the $j$-set from which to continue the path (this will be explained
in more detail in Section~\ref{sec:dfsalg}).

\subsection{Paper overview}

The remainder of the paper is arranged as follows.

In Section~\ref{sec:prelim}, we will analyse the structure of $j$-tight paths
and prove some preliminary results concerning the number of automorphisms,
which will be needed later. We also collect some standard probabilistic results
which we will use.

Subsequently, Section~\ref{sec:secondmoment} will be devoted to a second moment calculation,
which will be used to prove the lower bound on $L$ in the subcritical case of Theorem~\ref{thm:mainresult}.
This is in essence a very standard method, although this particular application presents considerable
technical challenges.

The second moment method breaks down when the paths become too long, and in particular it
is too weak to prove the lower bounds in the supercritical case.
Therefore the main contribution of this paper is an alternative
strategy, inspired by previous proofs of phase transition
results regarding the order of the giant component. These proofs, due to Krivelevich and Sudakov~\cite{KrivelevichSudakov13}
as well as Cooley, Kang, and Person~\cite{CooleyKangPerson18} and
Cooley, Kang, and Koch~\cite{CooleyKangKoch18}, are based on an analysis of search processes which explore components.

We therefore introduce
the \pathfinder\ algorithm, which is in essence a depth-first search process for paths,
in Section~\ref{sec:dfs}.
In Section~\ref{sec:algprops}, we observe some basic facts about the \pathfinder\ algorithm,
which we subsequently use in Section~\ref{sec:loose} ($j=1$)
and Section~\ref{sec:alganalysis} ($j\ge 2$) to prove that \whp\ the
\pathfinder\ algorithm finds a $j$-tight path of the appropriate length, proving the lower bounds on $L$
in the supercritical case of Theorem~\ref{thm:mainresult}.

We collect together all of the previous results to complete the proof of
Theorem~\ref{thm:mainresult} in Section~\ref{sec:mainproof}.
Finally in Section~\ref{sec:concluding} we discuss some open problems,
including possible strengthenings of Theorem~\ref{thm:mainresult}.

\section{Preliminaries}\label{sec:prelim}

We first gather some notation and terminology which we will use throughout the paper.

Throughout the paper, $k$ and $j$ are fixed integers with $1\le j \le k-1$.
All asymptotics are with respect to $n$, and we use the standard
Landau notations $o(\cdot)$, $O(\cdot),\Theta(\cdot),\Omega(\cdot)$
with respect to these asymptotics.
In particular, any value
which is bounded by a function of $k$ and $j$ is $O(1)$. If $S$ is a set and $m\in\NN_0,$ then $\binom{S}{m}$ denotes the set of $m$-element subsets of $S$. 
For $m,i \in \NN$, we use $(m)_i := m(m-1)\ldots(m-i+1)$ to denote the $i$-th falling factorial.

Recall that for $\ell\in\NN$, a $j$-tight path of length $\ell$ in a $k$-uniform hypergraph contains
$\ell$ edges and $(k-j)\ell + j$ vertices. Throughout the paper,
whenever $j,k,\ell$ are clear from the context, we will denote
by
\begin{equation}\label{eq:def:v}
v=v_{j,k}(\ell):=(k-j)\ell+j
\end{equation}
the number of vertices in such a path.
Furthermore, for the rest of the paper we fix $a$ as in Theorem~\ref{thm:mainresult},
i.e.\ $a$ is the unique integer such that
\begin{equation}\label{eq:def:a}
1\le a \le k-j \qquad \mbox{and} \qquad a \equiv k \pmod {k-j}
\end{equation}
and we set
\begin{equation}\label{eq:def:b}
b:=k-j-a.
\end{equation}

Throughout the paper we ignore floors and ceilings whenever these
do not significantly affect the argument.
For the sake of clarity and readability,
we delay many proofs
of auxiliary results, particularly those that are applications
of standard ideas or involve lengthy technical details, to the appendices.

\subsection{Structure of $j$-tight paths}\label{sec:pathstructure}

For $\ell\in\NN$, let $\cP_{\ell}$ be the set of all $j$-tight paths of length $\ell$ in the complete $k$-uniform hypergraph on $[n]$, denoted by $K^{(k)}_n$. Thus
$\cP_\ell$ is the set of \emph{potential} $j$-tight paths
of length $\ell$ in $H^k(n,p)$.

It is important to observe that, depending on the values of $k$ and $j$,
the presence of one $j$-tight path $P\in \cP_\ell$ in $H^k(n,p)$ may instantly
imply the presence of many more with exactly the same edge set.
In the graph case, there are only two paths with exactly the same edge set
(we obtain the second by reversing the orientation), but for general $k$
and $j$ there may be more.

Let us demonstrate this with the following example for the case $k=5$ and $j=2$ (see Figure~\ref{fig:pathpartition}).
\vspace{0.4cm}
\begin{center}
\begin{figure}[h]

\begin{tikzpicture}[level/.style={},decoration={brace,mirror,amplitude=7},scale=1.1]

\fill  (0,0)  circle (0.06)  ;
\fill  (0.6,0) circle (0.06) ;
\fill (1.2,0) circle (0.06);
\fill (1.8,0) circle (0.06);
\fill (2.4,0) circle (0.06);
\fill (3,0) circle (0.06);
\fill (3.6,0) circle (0.06);
\fill (4.2,0) circle (0.06);
\fill (4.8,0) circle (0.06);
\fill (5.4,0) circle (0.06);
\fill (6,0) circle (0.06);
\fill (6.6,0) circle (0.06);
\fill (7.2,0) circle (0.06);
\fill (7.8,0) circle (0.06);
\fill (8.4,0) circle (0.06);
\fill (9,0) circle (0.06);
\fill (9.6,0) circle (0.06);

\draw (1.18,0) ellipse (45pt and 12pt);
\draw[thick, dotted] (2.99,0) ellipse (45pt and 12pt);
\draw (4.77,0) ellipse (45pt and 12pt);
\draw[thick, dotted] (6.58,0) ellipse (45pt and 12pt);
\draw (8.4,0) ellipse (45pt and 12pt);

\draw [thick, black,decorate,decoration={brace,amplitude=3pt,mirror},xshift=0.4pt,yshift=-0.4pt](0,-0.7) -- (1.2,-0.7) node[black,midway,yshift=-0.6cm] {$F_1$};
\draw [thick, black,decorate,decoration={brace,amplitude=3pt,mirror},xshift=0.4pt,yshift=-0.4pt](1.75,-0.7) -- (2.45,-0.7) node[black,midway,yshift=-0.6cm] { $A_1$};
\draw [thick, black,decorate,decoration={brace,amplitude=2pt,mirror},xshift=0.4pt,yshift=-0.4pt](2.8,-0.7) -- (3.2,-0.7) node[black,midway,yshift=-0.6cm] {$B_1$};
\draw [thick, black,decorate,decoration={brace,amplitude=2pt,mirror},xshift=0.4pt,yshift=-0.4pt](3.6,-0.7) -- (4.2,-0.7) node[black,midway,yshift=-0.6cm] { $A_2$};
\draw [thick, black,decorate,decoration={brace,amplitude=2pt,mirror},xshift=0.4pt,yshift=-0.4pt](4.6,-0.7) -- (5,-0.7) node[black,midway,yshift=-0.6cm] { $B_2$};
\draw [thick, black,decorate,decoration={brace,amplitude=2pt,mirror},xshift=0.4pt,yshift=-0.4pt](5.3,-0.7) -- (6.1,-0.7) node[black,midway,yshift=-0.6cm] { $A_3$};
\draw [thick, black,decorate,decoration={brace,amplitude=2pt,mirror},xshift=0.4pt,yshift=-0.4pt](6.4,-0.7) -- (6.8,-0.7) node[black,midway,yshift=-0.6cm] { $B_3$};
\draw [thick, black,decorate,decoration={brace,amplitude=2pt,mirror},xshift=0.4pt,yshift=-0.4pt](7.2,-0.7) -- (7.9,-0.7) node[black,midway,yshift=-0.6cm] {$A_4$};
\draw [thick, black,decorate,decoration={brace,amplitude=3pt,mirror},xshift=0.4pt,yshift=-0.4pt](8.3,-0.7) -- (9.7,-0.7) node[black,midway,yshift=-0.6cm] { $G_1$};
\end{tikzpicture}
\caption{A $2$-tight path of length $5$ in a $5$-uniform hypergraph, with a natural
partition of vertices.}\label{fig:pathpartition}
\end{figure}
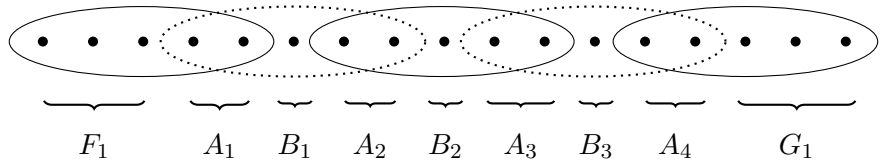

\end{center}

Observe that we have partitioned the vertices into sets ($F_1,A_1,\ldots$)
 according
to which edges they are in---each set of the partition is maximal with the property that every
vertex in that set is in exactly the same edges of the $j$-tight path.
Therefore we can re-order the vertices arbitrarily within any of
these sets and obtain another $j$-tight path with the same edge set,
and therefore also the same length.
Similarly as for graphs, we can also reverse the orientation
of the vertices (and also the edges) to obtain another $j$-tight
path with the same edge set.

It will often be convenient to consider such paths as being
the same, even though the order of vertices is different.
Therefore we define an equivalence relation $\sim_\ell$
on $\cP_{\ell}$ as follows. For any $A,B\in\cP_{\ell}$,
we say that $A \sim_\ell B$ if they have exactly the same edges.

We will be interested in the equivalence classes of this relation.
Let $z_\ell=z_\ell(k,j)$ denote the size of each equivalence class of $\sim_\ell$
(note that, by symmetry, each equivalence class has the same size
and so $z_\ell$ is well-defined).
Further, let $\hat{\cP}_{\ell}$ be the set of equivalence classes
of $\sim_\ell$.
Observe that if some $P \in \cP_\ell$ is in
$H^k(n,p)$, then so is every path in its equivalence
class $\hat P \in \hat \cP_\ell$. We abuse terminology slightly by saying
that the equivalence class $\hat P$ lies in $H^k(n,p)$,
and write $\hat P \subset H^k(n,p)$.
We define 
$\hat{X}_{\ell}$ to be the number of equivalence classes
for which this is the case.
Then
\begin{equation}\label{eqn:expectationexact}
    \EE(\hat{X}_{\ell})
    =\sum_{\hat P\in\hat{\cP}_{\ell}}\prob\left(\hat P \subset H^k(n,p)\right)
    =|\hat{\cP}_{\ell}|p^{\ell}=\frac{(n)_{v}}{z_\ell}p^{\ell},
\end{equation}
where $v=(k-j)\ell+j$ is the number of vertices in a $j$-tight path with $\ell$ edges (as defined in \eqref{eq:def:v}).

We therefore need to estimate $z_{\ell}$.
To do so, we will analyse the structure of $j$-tight paths, inspired by the
example in Figure~\ref{fig:pathpartition}.
This analysis leads to the following lemma.

\begin{lemma}\label{lem:isomorphisms}
Let $s=s(j,k):=\left\lceil\frac{k}{k-j}\right\rceil-1$. Then
\begin{equation*}
z_\ell=
\begin{cases}
\Theta(1) & \mbox{if } \ell \le s+1;\\
\frac{2}{b!}(a!b!)^{\ell-s}((k-j)!)^{2s} & \mbox{if } \ell \ge s+2.
\end{cases}
\end{equation*}
In particular,
\begin{equation}\label{eq:isomorphisms}
z_\ell= \Theta\left((a!b!)^\ell\right).
\end{equation}
\end{lemma}

\begin{proof}
Let us first observe that if $\ell \le s+1$, then
a $j$-tight path with $\ell$ edges has $v$ vertices,
where
$$v = (k-j)\ell+j \le k(\ell+1) \le k(s+2) = O(1),$$
and therefore $1\le z_\ell \le v! = O(1)$,
and the statement of the lemma follows for this case.
We therefore assume that $\ell \ge s+2$.

We aim to determine the natural partition of the vertices of a $j$-tight path 
according to which edges they are in, as we did in the example in Figure~\ref{fig:pathpartition}.

Denote the edges of the $j$-tight path $P \in \cP_\ell$ by $(e_1,\ldots,e_{\ell})$,
in the natural order.
Recall that $s=\lceil \frac{k}{k-j}\rceil-1$, and observe that $s$ is the largest
integer such that $(k-j)s <k$, and therefore the largest integer
such that $e_i\cap e_{i+s} \neq \emptyset$.
We define
\begin{align*}
F_i & :=e_i\setminus e_{i+1} &  \mbox{for } 1\le i \le s;\\
G_i & :=e_{\ell-s+i}\backslash e_{\ell-s+i-1} & \mbox{for }1\le i \le s.
\end{align*}
We also define
\begin{align*}
A_i&:=e_i\cap e_{i+s} & \mbox{for } 1\leq i \leq \ell-s,\hspace{0.68cm}\\
B_i&:=e_{i+s}\backslash (e_{i+s+1}\cup e_i) &  \mbox{for } 1\leq i\leq\ell-s-1.
\end{align*}

Observe that $A_i \cup B_i = e_{i+s} \setminus e_{i+s+1}$.
Furthermore, since $s$ is the largest integer such that $e_{i+s+1}$ intersects $e_{i+1}$, we have that $(e_{i+s}\setminus e_{i+s+1})\subset e_{i+1}$
and that $A_{i+1} \subseteq (e_{i+1}\setminus e_{i})$,
and therefore $A_{i+1}\cup B_i=e_{i+1}\setminus e_{i}$.
Since we also have $A_i \cap B_i = A_{i+1} \cap B_i = \emptyset$,
the vertices of the path $P$ are now partitioned into parts
$$
(F_1,\ldots,F_s,A_1,B_1,A_2,B_2,\ldots,A_{\ell-s-1},B_{\ell-s-1},A_{\ell-s},G_1,\ldots,G_s)
$$
(in the natural order along $P$).
Observe further that
the parts are of maximal size
such that the vertices within each part are in exactly the same edges.
We refer to $\bigcup_{i=1}^s F_i = e_1\setminus e_{s+1}$ as the \emph{head}
of the path $P$ and to $\bigcup_{i=1}^s G_i = e_\ell\setminus e_{\ell-s}$
as the \emph{tail}.
Note that the vertices within each part can be rearranged to obtain
a new $j$-tight path with exactly the same edges.
We can also change the orientation of the path (i.e.\ reverse the order of
the edges) to obtain a new path with the same edge set.
(If $\ell=0,1$, this reorientation would already have been counted,
but recall that we have assumed that $\ell \ge s+2$.)
Thus we have
\begin{equation}\label{eq:reorderings}
z_\ell = 2 \left(\prod_{i=1}^s |F_i|!|G_i|!\right)\left(\prod_{i=1}^{\ell-s}|A_i|!\right)\left(\prod_{i=1}^{\ell-s-1}|B_i|!\right).
\end{equation}
It therefore remains to determine the sizes of the $F_i,G_i,A_i,B_i$.

\begin{claim}\label{claim:partsizes}
\begin{align*}
|F_i|=|G_i| & =k-j & \mbox{for }1\le i \le s;\\
|A_i| & =a & \mbox{for }1\le i \le \ell-s;\\
|B_i| & =b & \mbox{for }1\le i \le \ell-s-1.
\end{align*}
\end{claim}

Substituting these values into~\eqref{eq:reorderings},
we obtain precisely the statement of Lemma~\ref{lem:isomorphisms}.
Thus the proof is complete up to verifying Claim~\ref{claim:partsizes}.
This proof, which consists of an elementary checking of the definitions,
appears in Appendix~\ref{app:partsizes}.
\end{proof}

Equation~\eqref{eqn:expectationexact} and Lemma~\ref{lem:isomorphisms} together
give the following immediate corollary.

\begin{corollary}\label{cor:expectationupperbound}
$$
\expec(\hat{X}_{\ell})=\Theta(1)\frac{(n)_{v}}{(a!b!)^{\ell}}p^{\ell}.
$$
\end{corollary}

\subsection{Large deviation bounds}

In this section we collect some standard results which will be needed later.

We will use the following Chernoff bound,
(see e.g.~\cite[Theorem 2.1]{JansonLuczakRucinskibook}).
We use $\Bi(N,p)$ to denote the binomial distribution with parameters $N \in \NN$ and $p \in [0,1]$.
\begin{lemma}\label{lem:chernoffbounds}
	If $X\sim \Bi(N,p)$, then for any $\xi\ge 0$
	\begin{equation} \label{eqn:chernoffstd}
	\mathbb{P}(X\geq Np+\xi)\leq\exp\left(-\frac{\xi^2}{2(Np+\frac{\xi}{3})}\right),
	\end{equation}
	and
	\begin{equation*}
	\mathbb{P}(X \leq Np - \xi)\leq\exp\left(-\frac{\xi^2}{2Np}\right).
	\end{equation*}
	
\end{lemma}

It will often be more convenient to use the following one-sided
form, which follows directly from Lemma~\ref{lem:chernoffbounds}.
The proof appears in Appendix~\ref{app:chernoff}.
\begin{lemma}\label{lem:Chernoffwitherror}
Let $X\sim \Bi(N,p)$ and let $\alpha>0$ be some arbitrarily small constant.
Then with probability at least $1-\exp(-\Theta(n^\alpha))$ we have
	$X \le 2Np+n^\alpha$.
\end{lemma}

\section{Second moment method: lower bound} \label{sec:secondmoment}

In this section we prove the lower bound in statement~\ref{item:thm-subcritical} of Theorem~\ref{thm:mainresult}.
The general basis of the argument is a completely standard second moment method--- 
however, applying the method to this particular problem is rather tricky and so the argument is lengthy.

For technical reasons that will become apparent during the proof,
we need to handle the case when $2\le j = k-1$ slightly differently.
We therefore distinguish two cases:
\begin{itemize}
\item Case 1: Either $j\le k-2$ or $j=k-1=1$.
\item Case 2: $2\le j = k-1$.
\end{itemize}
 Correspondingly, we split the lower bound we aim to prove into two lemmas.
In Case~1, we need to prove the following.

\begin{lemma}\label{lem:lowerbound-subcritical:case1}
	Let $k,j\in\NN$ satisfy $1 \leq j \leq k-1$, and additionally
	either $j\le k-2$ or $j=k-1=1$. Let $a\in\NN$ be the unique integer satisfying
	$1 \leq a \leq k-j$ and $a \equiv k \bmod (k-j)$.
	Let $\eps = \eps(n) \ll 1$ satisfy $\eps^3 n \xrightarrow{n\to\infty} \infty$
	and let
	$$
	p= \frac{1-\eps}{\binom{k-j}{a} \binom{n-j}{k-j}}.
	$$
	Let $L$ be the length of the longest $j$-tight path in $H^k(n,p)$.
Then \whp\
\[
L \geq \frac{j \ln n- \omega + 3 \ln \eps}{-\ln(1-\eps)},
\]
for any $\omega=\omega(n)$ such that $\omega\xrightarrow{n\to\infty} \infty$.
\end{lemma}

On the other hand, in Case~2 we have $k-j=1$, and therefore
the parameter $a$ from Theorem~\ref{thm:mainresult} is simply $1$.
Thus also $\binom{k-j}{a}=1$ and $p_0=\frac{1}{n-k+1}$,
and so the lower bound in Theorem~\ref{thm:mainresult}~\ref{item:thm-subcritical}
simplifies to the following.

\begin{lemma}\label{lem:lowerbound-subcritical:case2}
	Let $k,j\in\NN$ satisfy $2 \leq j = k-1$.
	Let $\eps = \eps(n) \ll 1$ satisfy $\eps^3 n \xrightarrow{n\to\infty} \infty$
	and let
	$$
	p= \frac{1-\eps}{n-k+1}.
	$$
	Let $L$ be the length of the longest $j$-tight path in $H^k(n,p)$.
Then \whp\
\[
L \geq \frac{j \ln n- \omega + 3 \ln \eps}{-\ln(1-\eps)},
\]
for any $\omega=\omega(n)$ such that $\omega\xrightarrow{n\to\infty} \infty$.
\end{lemma}

Since the main ideas in the proofs of these two lemmas are essentially identical, we will
treat only Case~1 (i.e.\ Lemma~\ref{lem:lowerbound-subcritical:case1}) here
and address Case~2 (i.e.\ Lemma~\ref{lem:lowerbound-subcritical:case2}) in Appendix~\ref{app:secondmomentcase2}.

\subsection{Case 1: Either $j\le k-2$ or $j=k-1=1$}
We will prove Lemma~\ref{lem:lowerbound-subcritical:case1} with the help
of various auxiliary results. Since these results are rather technical in nature,
we also defer their proofs to Appendix~\ref{app:secondmomentcase1}.

Let us set $\ell=\frac{j \ln n- \omega + 3 \ln \eps}{-\ln(1-\eps)}$.

Recall that $\cP_{\ell}$ is the set of all $j$-tight paths
of length $\ell$ in $K^{(k)}_n$,
and therefore
$$
\EE(X_\ell^2)=\sum_{A,B\in \cP_\ell} \Pr(A,B\subset H^k(n,p)).
$$
The probability term in the sum is fundamentally dependent
on how many edges the paths $A$ and $B$ share, so we will
need to calculate the number of pairs of possible paths
with given intersections.

For any $A,B\in\cP_{\ell}$,
let $Q(A,B)$ be the set of common edges of $A$ and $B$ and define $q(A,B):=|Q(A,B)|$.
Observe that there is a natural partition of $Q(A,B)$ into
\emph{intervals}, where each interval is a maximal set
of edges in $Q(A,B)$ which are consecutive along
both $A$ and $B$.
Let $r(A,B)$ be the number of intervals in this natural partition of $Q(A,B)$.
Set $\mathbf{c}(A,B):=(c_1,\ldots,c_r)$,
where $c_1 \geq \dotsb \geq c_r \geq 1$, to be the lengths (i.e.\ the number of edges)
of these intervals.
Given non-negative integers $q,r$ and an $r$-tuple $\mathbf{c}=(c_1,\ldots,c_r)$
such that $c_1\geq\dotsb\geq c_r\geq 1$ and $c_1+\dotsb+c_r=q$, define
\begin{align*}
\cP_{\ell}^2(q,r,\mathbf{c}):=\{(A,B)\in\cP_{\ell}^2:\ q(A,B) & =q, \\
r(A,B)& =r, \\
\mathbf{c}(A,B) & =\mathbf{c}\}.
\end{align*}

For any $q,r,\mathbf{c}$ not satisfying these conditions, $\cP_\ell^2(q,r,\mathbf{c})$
is empty. Recall from~\eqref{eq:def:v} that $v=(k-j)\ell+j$ is the number of vertices in a $j$-tight path of length~$\ell$. 
\begin{claim}\label{claim:expectationfirsttermsplit}
\begin{equation}\label{eqn:expectationfirsttermsplit}
\EE(X_\ell^2) \le \left((n)_v\right)^2p^{2\ell} + \sum_{q\ge 1}\sum_{r\ge 1}\sum_{\mathbf{c}}|\cP_\ell^2(q,r,\mathbf{c})|p^{2\ell-q}.
\end{equation}
\end{claim}

Thus we need to estimate $|\cP_{\ell}^2(q,r,\mathbf{c})|$ for $q,r\ge 1$.
Given $q,r\ge 1$, we define the parameter
\[
T(r)=T_q(r):= 
(k-j)q+j+(r-1)\min\{j,k-j\}.
\]
This slightly arbitrary-looking expression is in fact a lower bound on
the number of vertices in $Q(A,B)$, as will become clear in the proof.
We obtain the following.

\begin{proposition}\label{prop:pl2qrcbound}
There exists a constant $C>0$ such that for any $q\ge 1$ we have
\begin{align*}
|\cP_\ell^2(q,r,\mathbf{c})| & \le \left((n)_{v}\right)^2 \frac{(\ell-q+1)^2 \ell^{2(r-1)}(a!b!)^q C^r}{(n-v)^{T(r)}}.
\end{align*}
\end{proposition}

Proposition~\ref{prop:pl2qrcbound} together with~\eqref{eqn:expectationfirsttermsplit} gives the
following immediate corollary.

\begin{corollary}\label{cor:squareexp}
There exists a constant $C>0$ such that
\begin{equation}\label{eqn:squareexp}
\EE(X_\ell^2) 
\leq \left((n)_{v}\right)^2p^{2\ell}\left(1+ \sum_{q=1}^{\ell}\ \sum_{r=1}^{q}\ \sum_{\substack{c_1+\dotsb+c_r=q \\ c_1\geq\dotsb\geq c_r\geq 1}}\frac{(\ell-q+1)^2\ell^{2(r-1)} (a!b!)^qC^r}{p^q (n-v)^{T(r)}}\right).
\end{equation}
\end{corollary}

We bound the triple-sum using the following two results.

\begin{proposition}\label{prop:doublesum}
\begin{equation}\label{eq:doublesum}
\sum_{r=1}^{q}\ \sum_{\substack{c_1+\dotsb+c_r=q \\ c_1\geq\dotsb\geq c_r\geq 1}}\frac{(\ell-q+1)^2\ell^{2(r-1)} (a!b!)^qC^r}{p^q (n-v)^{T(r)}}
=
O\left(n^{-j}\right) \frac{(\ell-q+1)^2}{(1-\eps)^q}.
\end{equation}
\end{proposition}

\begin{claim}\label{claim:roneterm}
\begin{align}
\sum_{q=1}^{\ell}\frac{(\ell-q+1)^2}{(1-\eps)^q} 
& = \frac{2(1-\eps)^{-\ell}}{\eps^3}.\label{eqn:roneterm}
\end{align}
\end{claim}

Substituting~\eqref{eq:doublesum} and~\eqref{eqn:roneterm} into~\eqref{eqn:squareexp},
using the fact that $\ell=\frac{j\ln n-\omega + 3 \ln \eps}{-\ln(1-\eps)}$ and performing some elementary
approximations leads to the following.

\begin{claim}\label{claim:secondmoment}
$\EE(X_\ell^2) =  \left((n)_{v}\right)^2p^{2\ell}(1+o(1))$.
\end{claim}

We can now use these auxiliary results to prove our lower bound.

\begin{proof}[Proof of Lemma~\ref{lem:lowerbound-subcritical:case1}]
Recalling that $\cP_\ell$ is the set of all possible $j$-tight paths of length $\ell$ in $H^k(n,p)$,
clearly $\EE(X_\ell) = |\cP_\ell| p^\ell = (n)_v p^\ell$.
Therefore by Claim~\ref{claim:secondmoment}, we have
$$
\EE(X_\ell^2) = \EE(X_\ell)^2(1+o(1)),
$$
and a standard application of Chebyshev's inequality shows that whp $X_\ell \ge 1$,
i.e.\ whp
\[L(G(n,p))\geq \ell = \frac{j\ln n - \omega + 3 \ln \eps}{-\ln(1-\eps)}\]
as claimed.
\end{proof}

It would be tempting to try to generalise this proof to also prove
a lower bound in the supercritical case. However, this strategy fails
because as the paths $A$ and $B$ become longer, there are many more
ways in which they can intersect each other, and therefore the terms which,
in the subcritical case, were negligible lower order terms (i.e.\ $q\ge 1$)
become more significant. We will therefore use an entirely different strategy
for the supercritical case.

\section{The \pathfinder\ algorithm}\label{sec:dfs}

The proof strategy for the lower bound in the supercritical case
is to define a depth-first search algorithm, which we call
\pathfinder\ and which discovers $j$-tight paths
in a $k$-uniform hypergraph, and to show that whp this algorithm, when applied to $H^k(n,p)$, will find a path of the
appropriate length.

\subsection{Special case: tight paths in $3$-uniform hypergraphs}\label{sec:specialdfs}

Before introducing the \pathfinder\ algorithm,
we briefly describe the algorithm in the special case $k=3$
and $j=2$, in order to introduce some of the ideas required
for the more complex general version.

In the special case, given a $3$-uniform hypergraph $H$,
the algorithm aims to construct a tight
path in $H$ starting at some \emph{ordered pair} of vertices $(v_1,v_2)$.
It will maintain a partition of the (unordered) pairs into \emph{neutral},
\emph{active}, and \emph{explored} pairs; initially only $\{v_1,v_2\}$
is active and all other pairs are neutral.

The algorithm now runs through the remaining $n-2$ vertices
(apart from $v_1,v_2$) in turn, for each such vertex $x$
making a \emph{query} to reveal whether $\{v_1,v_2,x\}$ forms an
edge of $H$. If we do not find
such an edge, then the pair $\{v_1,v_2\}$ is labelled explored,
and we choose a new ordered pair from which to begin
(the corresponding unordered pair is then labelled active, and the corresponding vertices take
the place of $v_1,v_2$). On the other hand, if
we do find an edge $\{v_1,v_2,x\}$, then we set $v_3=x$, label
the pair $\{v_2,v_3\}$ active and look for ways to extend the
path from this pair.

More generally, at each step of the algorithm the current
path will consist of vertices $v_1,v_2,\ldots,v_{\ell+2}$,
where $\ell$ is the length (i.e.\ number of edges of the path).
The set of active pairs will consist of $\{v_i,v_{i+1}\}$ for $1\le i \le \ell+1$,
and we will seek to extend the path from $\{v_{\ell+1},v_{\ell+2}\}$.
We therefore aim to query triples $\{v_{\ell+1},v_{\ell+2},x\}$,
but we have some restrictions on when such a query can be made:
\begin{enumerate}[\textnormal{\arabic*.}]
\item $\{v_{\ell+1},v_{\ell+2},x\}$ must not have been queried from $\{v_{\ell+1},v_{\ell+2}\}$ before;
\item $x$ may not lie in $\{v_1,\ldots,v_{\ell+2}\}$;
\item Neither $\{v_{\ell+1},x\}$ nor $\{v_{\ell+2},x\}$ may be explored.
\end{enumerate}

The purpose of the first condition is clear: this ensures
that we do not repeat previous queries and get stuck in a loop.
The second condition forbids extensions which re-use a vertex
which is already in the current path, which is also clearly necessary.

The third condition is perhaps the most interesting one.
The algorithm would run correctly and find a tight path
even without this condition,
but it
does ensure that no triple is ever queried more than once,
which might otherwise occur as the triple $\{v_{\ell+1},v_{\ell+2},x\}$ might have been queried
from, say, the explored pair $\{v_{\ell+1},x\}$. While this
would be permissible to create a new tight path, it
would mean that the outcomes of some queries are dependent
on each other, making the \emph{analysis} of the algorithm
far more difficult.

We therefore forbid such queries, which means that we may
not find the longest path in the hypergraph, but if we
still find a path of the required length, this is sufficient.

If we find an edge $\{v_{\ell+1},v_{\ell+2},x\}$ from the pair $\{v_{\ell+1},v_{\ell+2}\}$,
we set $v_{\ell+3}=x$, label $\{v_{\ell+2},v_{\ell+3}\}$ active
and continue exploring from this pair.
If on the other hand we find no such edge from $\{v_{\ell+1},v_{\ell+2}\}$,
then we label $\{v_{\ell+1},v_{\ell+2}\}$ explored, remove $v_{\ell+2}$
from the path and continue exploring from the previous
active pair, i.e.\ $\{v_\ell,v_{\ell+1}\}$
(unless $\ell=0$ in which case we have no further active
pairs and we pick a new, previously neutral pair to start from,
and order the vertices of this pair arbitrarily).

We now highlight a few ways in which the algorithm
for general $k$ and $j$ differs from this special
case, before introducing the algorithm more formally in
Section~\ref{sec:dfsalg}.

Rather than the \emph{pairs} of vertices, it will be the \emph{$j$-sets}
of vertices which are neutral, active or explored.
We also begin our exploration process from a $j$-set
rather than a pair.

In the special case, we also had an order of the vertices,
and began with an ordered pair. In general, we will not
necessarily have a \emph{total} order of the vertices in the path,
but we will have a \emph{partial} order, or more precisely an
\emph{ordered partition} of each $j$-set into
a set of size $a$ and some sets of size $k-j$.
This is connected to the fact that the last active
$j$-set in the current path will contain the tail (see Section~\ref{sec:pathstructure}),
and the ordered partition specifies which vertices
belong to the sets $G_1,\ldots,G_s$.

Related to this, depending on the values of $k$ and $j$,
when we discover an edge $K$ from a $j$-set $J$,
it may be that more than one
$j$-set becomes active. More precisely, the tail will
shift from $G_1,\ldots,G_s$ to $G_2,\ldots,G_{s+1}$,
where $G_{s+1}=K\setminus J$, and any $j$-set containing
the new tail and $a$ vertices from $G_1$ is a valid place
to continue extending the path, and therefore becomes active.

A consequence of this is that $j$-sets become active
in \emph{batches} of size $\binom{k-j}{a}$.
Such a batch becomes active each time we discover an edge
and from any $j$-set of the batch we can continue the path.
Therefore we do not remove an edge (and decrease the length
of the path) every time a $j$-set becomes explored---we only
do this once \emph{all} $j$-sets of the corresponding batch have become explored.

\subsection{Hypergraph exploration using DFS}\label{sec:dfsalg}

In this section, we will describe the \pathfinder\ algorithm to find $j$-tight paths in $k$-uniform
hypergraphs in full generality. 
We will use the following notation:
if $\cF$ is a family of sets and $X$ is a set, we write $\cF+X$ and $\cF-X$ to mean $\cF \cup \{X\}$ and $\cF \setminus \{X\}$ respectively.

Recall from~\eqref{eq:def:a} that $a \in [k-j]$ is such that $a \equiv k \bmod (k-j)$,
and from the statement of Lemma~\ref{lem:isomorphisms} that $s=\lceil \frac{k}{k-j}\rceil-1= \lceil \frac{j}{k-j}\rceil$.
Let us define $r:=s-1 = \lceil\frac{j}{k-j}\rceil-1$,
so that $j = a+(k-j)r$.
\begin{definition}
Given a set $J$ of $j$ vertices, an \emph{extendable partition of $J$}
is an ordered partition $(C_0, C_1, \dotsc, C_r)$ of $J$
such that $|C_0| = a$ and $|C_i| = k-j$ for all $i \in [r]$.
\end{definition}

Note that if we have constructed a reasonably long $j$-tight path (i.e.\ of length at least $s$),
the final $j$-vertices naturally come with an extendable partition $(C_0,C_1,\ldots,C_r)$ according to which edges
of the path they lie in, similar to the partition of all vertices of the path described
in Section~\ref{sec:pathstructure}. The vertices within each part of the extendable partition
could be re-ordered arbitrarily to obtain a new path with the same edge set.
Therefore if we find a further edge from the final $j$-set to extend the path, there is more than one possibility
for the final $j$-set of the extended path---it must contain $C_2,\ldots,C_r$ and a further $a$ vertices from $C_1$, which may be chosen arbitrarily.
Thus an extendable partition provides a convenient way to describe the $j$-sets from which we might
further extend the path.

Although for paths of length shorter than $s$ the $j$-sets come only with a coarser (and therefore less restrictive)
partition, it is convenient for a unified description of the algorithm for them to be given an extendable partition.
In particular, we will start our search process from a $j$-set which we artificially endow with an extendable partition;
this additional restriction is permissible for a lower bound on the longest path length.

We begin by giving an informal overview of the algorithm---
the formal description follows.

\renewcommand{\thealgocf}{}

\SetAlFnt{\footnotesize}
\begin{algorithm}
	\DontPrintSemicolon
	\KwIn{Integers $k, j$ such that $1\le j \le k-1$.}
	\KwIn{$H$, a $k$-uniform hypergraph.}
	Let $a \in [k-j]$ be such that $a \equiv k \bmod (k-j)$\;
	Let $r=\lceil \frac{j}{k-j}\rceil-1$\;
	For $i \in \{j,k\}$, let $\sigma_i$ be a permutation of the $i$-sets of $V(H)$, chosen uniformly at random\;
	$N \gets \binom{V(H)}{j}$ \tcp*{neutral $j$-sets}
	$A, E \gets \emptyset$ \tcp*{active, explored $j$-sets}
	$P \gets \emptyset$ \tcp*{current $j$-tight path}
	$\ell \gets 0$ \tcp*{index tracking the current length of $P$}
	$t \gets 0$ \tcp*{``time'', number of queries made so far}
	
	\While{$N \neq \emptyset$}{
		Let $J$ be the smallest $j$-set in $N$, according to $\sigma_j$ \tcp*{``new start''}
		Choose an arbitrary extendable partition $\cP_J$ of $J$\;
		$B_0 = \{ J \}$ \;
		$A \gets \{ J \}$\;
		\While{$A \neq \emptyset$}{
			Let $J$ be the last $j$-set in $A$\;
			Let $\cK$ be the set of $k$-sets $K \subset V(H)$ such that
				$K \supset J$,
				$K$ was not queried from $J$ before,
				$K\setminus J$ is vertex-disjoint from $P$,
				and $K$ does not contain any $J' \in E$\;
			\uIf{$\cK \neq \emptyset$}{
				Let $K$ be the first $k$-set in $\cK$ according to $\sigma_k$\;
				$t \gets t+1$ \tcp*{a new query is made}
				\If(\tcp*[f]{``query $K$''}){$K \in H$}{
					$e_\ell \gets K$ \;
					$P \gets P + e_{\ell}$ \tcp*{$P$ is extended by adding $K=e_{\ell}$}
					$\ell \gets \ell+1$ \tcp*{length of $P$ increases by one} 
					Let $\cP_J=(C_0, C_1, \dotsc, C_r)$ be the extendable partition of~$J$\;
					\For{each $Z \in \binom{C_1}{a}$}{
						$J_Z \gets Z \cup C_2 \cup \dotsb \cup C_r \cup (K \setminus J)$ \tcp*{$j$-set to be added}
						$\cP_{J_Z} \gets (Z, C_2, \dotsc, C_r, K \setminus J)$\tcp*{extendable partition}
						$i(J_Z) \gets \ell$ \;
						$A \gets A + J_Z$ \tcp*{$j$-set becomes active}
					}
					$\cB_{\ell} \gets \{ J_Z : Z \in \binom{C_1}{a} \}$\;
				}
				$(A_t, E_t, P_t) \gets (A, E, P)$ \tcp*{update ``snapshot'' at time $t$}
			}
			\ElseIf(\tcp*[f]{all extensions from $J$ were queried}){$\cK = \emptyset$}{
				$A \gets A - J$ \tcp*{$J$ becomes explored}
				$E \gets E + J$\;
					\If(\tcp*[f]{the current batch is fully explored}){ $\cB_\ell \subset E$ }{
						$\cB_\ell \gets \emptyset$ \tcp*{empty this batch}
						$P \gets P - e_{\ell}$ \tcp*{last edge of $P$ is removed}
						$\ell \gets \ell - 1$ \tcp*{length of $P$ decreases by one}
					}
				}
			}
		}
	\caption{\pathfinder}
	\label{algorithm:dfs}
\end{algorithm}

At any given point, the algorithm will maintain a $j$-tight path $P$
and a partition of the $j$-sets of $V(H)$ into \emph{neutral}, \emph{active} or \emph{explored} sets.
Initially, $P$ is empty and every $j$-set is neutral.
During the algorithm every $j$-set can change its status from neutral to active and from active to explored.
The $j$-sets which are active or explored will be referred to as \emph{discovered}.

The edges of $P$ will be $e_1, \dotsc, e_{\ell}$ (in this order),
and every active $j$-set will be contained inside some edge of $P$.
Whenever a new edge $e_{\ell + 1}$ is added to the end of $P$,
a batch $\cB_{\ell + 1}$ of neutral $j$-sets within that edge will become active:
these are the $j$-sets from which we could potentially extend the current path.
A $j$-set $J$ becomes explored once all possibilities to extend $P$ from $J$ have been queried.
Once all of the $j$-sets in the batch $\cB_{\ell}$ corresponding to $e_{\ell}$
have been declared explored, $e_{\ell}$ will be removed from $P$.

The active sets will be stored in a ``stack'' structure (last in, first out).
Each active $j$-set $J$ will have an associated extendable partition
$\cP_J$ of $J$, and an index $i(J) \in \{ 0, \dotsc, \ell \}$,
where $\ell$ is the current length of $P$.
The extendable partition will keep track of the ways in which we can extend $P$ from $J$
in a consistent manner, as described in Section~\ref{sec:dfsalg}.
The index $i(J)$ will indicate that $J$ belongs to the batch $\cB_{i(J)}$
which was added when the edge $e_{i(J)}$ was added to $P$.
Thus the algorithm will maintain a collection of batches $\cB_0, \dotsc, \cB_\ell$,
all of which consist of discovered $j$-sets which are inside $V(P)$.
It will hold that $|\cB_0| = 1$ and $|\cB_i| = \binom{k-j}{a}$ for all $i \geq 1$, and all the batches will be disjoint.

All the $j$-sets from a single batch will change their status 
from neutral to active in a single step, and they will be added to the stack according to some fixed order
which is chosen uniformly at random during the initialisation of the algorithm.

An iteration of the algorithm can be described as follows.
Suppose $J$ is the last active $j$-set in the stack.
We will query $k$-sets $K$, to check whether~$K$ is an edge in~$H$ or not.
We only query a $k$-set $K$ subject to the following conditions:
\begin{enumerate}[\textnormal{(Q\arabic*)}]
\item \label{cond:YcontJ} $K$ contains~$J$;
\item \label{cond:vertexrepeat} $K \setminus J$ is disjoint from the current path $P$;
\item \label{cond:queryrepeat} $K$ was not queried from $J$ before;
\item \label{cond:noexplored} $K$ does not contain any explored $j$-set.
\end{enumerate}
Condition~\ref{cond:YcontJ} ensures that we only query $k$-sets with which we might
sensibly continue the path in a $j$-tight manner.
Condition~\ref{cond:vertexrepeat} ensures that we do not
re-use vertices that are already in $P$.
Together, these two conditions guarantee that $P$ will indeed
always be a $j$-tight path.
Moreover, Condition~\ref{cond:queryrepeat} ensures that
we never query a $k$-set more than once from the same $j$-set,
thus guaranteeing that the algorithm does not get stuck in an infinite loop.
Finally Condition~\ref{cond:noexplored} ensures that we never query a $k$-set
a second time from a \emph{different} $j$-set (note that
the possibility that $K$ could have been queried from another \emph{active}
$j$-set is already excluded by Condition~\ref{cond:vertexrepeat},
since such an active $j$-set would lie within $P$).
Note that, as described in Section~\ref{sec:specialdfs},
Condition~\ref{cond:noexplored} is not actually necessary for the
correctness of the algorithm, but it does ensure independence of queries
and is therefore necessary for our \emph{analysis} of the algorithm.

If no such $k$-set~$K$ can be found in the graph $H$, then we declare~$J$
explored and move on to the previous active $j$-set in the stack.
Moreover, if at this point all of the $j$-sets in the batch
$\cB_{i(J)}$ of~$J$ have been declared explored, the last edge $e_\ell$
of the current path is removed and $\ell$ is replaced by $\ell-1$.
If the set of active $j$-sets is now empty, we choose a new $j$-set
$J$ from which to start, declare $J$ active and choose an extendable
partition of $J$.

On the other hand, if we can find a suitable set~$K$ for $J$,
we query $K$, and if it forms an edge, then
according to the extendable partition of $J$, the set~$K$ will yield
a new batch of $j$-sets (which previously were neutral and now become active).
More precisely, if the extendable partition of $J$ is
$(C_0,C_1,\ldots,C_r)$, then the
batch consists of all $j$-sets which contain $K\setminus J$ and
$C_2,\ldots,C_r$,
as well as $a$ vertices of $C_1$. Thus the batch
consists of $\binom{k-j}{a}$ many $j$-sets.

Finally, we keep track of a ``time'' parameter $t$,
which counts the number of queries the algorithm has made.
Initially, $t = 0$ and $t$ increases by one each time we query
a $k$-set.

During the analysis we will make reference to certain
objects or families which are implicit in the algorithm at each time $t$ even if
the algorithm does not formally track them.
These include the sets of neutral, active and discovered $j$-sets
$N_t,A_t,E_t$ and the current path $P_t$, which are simply the sets $N,A,E$ and the path $P$ at time $t$.
We say that $(A_t, E_t, P_t)$ is the \emph{snapshot of $H$ at time $t$}.
We also refer to certain families of $j$-sets, including $D_t$ (the discovered $j$-sets),
$R_t$ (the ``new starts'') and $S_t$ (the ``standard $j$-sets''),
as well as families $F^{(1)}_t, F^{(2)}_t, F_t$
of $(k-j)$-sets (the ``forbidden subsets'').
The precise definitions of all of these families will be given when they become relevant.

\subsection{Proof strategy}

Our aim is to analyse the \pathfinder\ algorithm and show that \whp\ it finds a path of length at least
$\frac{(1-\delta)\eps n}{(k-j)^2}$, or at least
$\frac{(1-\delta)\eps^2 n}{4(k-j)^2}$ if $j=1$. The overall strategy
can be described rather simply: suppose that by some
time $t$, which is reasonably large, we have not discovered
a path of the appropriate length.
Then \whp\ (and disregarding some small error terms),
the following holds:
\begin{enumerate}[(\Alph*)]
\item \label{step:edges} We have discovered at least $pt\binom{k-j}{a}$ many $j$-sets;
\item \label{step:explored} Very few $j$-sets are active, therefore at least $pt\binom{k-j}{a}$ are explored;
\item \label{step:queriesfromexplored} From each explored $j$-set, we queried at least $\binom{n'}{k-j}$ many $k$-sets,
where $n'= \left(1-\frac{(1-\delta)\eps}{k-j}\right)n$.
\item \label{step:contradiction} Thus the number of queries made is at least
\begin{align*}
pt\binom{k-j}{a}\binom{n'}{k-j}
& = t\frac{(1+\eps)}{\binom{n}{k-j}}\binom{\left(1-\frac{(1-\delta)\eps}{k-j}\right)n}{k-j}\\
& \approx t(1+\eps)(1-(1-\delta)\eps) >t.
\end{align*}
\end{enumerate}
This yields a contradiction since the number of queries made
is exactly $t$ by definition.

The proof consists of making these four steps more precise.
Three of these four steps are very easy to prove, once the
appropriate error terms have been added:

Step~\ref{step:edges}
follows from a simple Chernoff bound applied to the number of edges
discovered, along with the observation that for each edge, we
discover $\binom{k-j}{a}$ many $j$-sets.

Step~\ref{step:explored} follows from the observation
that all active $j$-sets lie within some edge of the current
path, and therefore there are at most $O(\eps n)$ of them,
which (for large enough $t$) is a negligible proportion
of the number of discovered $j$-sets, and therefore
almost all discovered $j$-sets must be explored.

Step~\ref{step:contradiction} is a basic calculation
arising from the bounds given by the previous three steps (though
in the formal proof we do need to incorporate some error
terms which we have omitted in this outline).

Thus the main difficulty is to prove Step~\ref{step:queriesfromexplored}.
Recall that a $k$-set $K$ containing $J$ may not be queried
for one of two reasons:
\begin{itemize}
\item $K\setminus J$ contains some vertex of $P$;
\item $K$ contains some explored $j$-set.
\end{itemize}
It is easy to bound the number of $k$-sets forbidden by the
first condition, since we assumed that the path was never long---
this is precisely what motivates the definition of $n'$.
However, we also need to show that \whp\ the number of $k$-sets
forbidden by the second condition is negligible, which will
be the heart of the proof.

\section{Basic properties of the algorithm}\label{sec:algprops}

Before analysing the likely evolution of the \pathfinder\ algorithm,
we first collect some basic properties which will be useful later.

Note that there are two ways in which a $j$-set $J$ can be discovered up to time $t$.
First, it could have been included as a \emph{new start} when the set of active
$j$-sets was empty and we chose a $j$-set $J$ from which to start exploring a new path (Line 10).
Second, $J$ could have been declared active if it was part of a batch
of $j$-sets activated when we discovered an edge, which we refer to
as a \emph{standard activation} (Lines 20-30), and we refer to the
$j$-sets which were discovered in this way as \emph{standard} $j$-sets.

For any $t \ge 0$, let $\ell_t := |E(P_t)|$ be the length (i.e.\ number of edges) of the path found by the algorithm at time $t$.

\begin{proposition}
At any time $t$, the number $|A_t|$ of active $j$-sets is at most
\begin{align}
|A_t| \leq 1 + \binom{k-j}{a} \ell_t. \label{eqn:activeviapath}
\end{align}
\end{proposition}

\begin{proof}

Recall that by construction, every active $j$-set in $A_t$ is contained in some edge of $P_t$.
Moreover, every time an edge is added to the current path,
exactly $\binom{k-j}{a}$ many $j$-sets are added via a standard activation.
There is also exactly one further active $j$-set which was added as a new start,
which gives the desired inequality.
\end{proof}
Note that equality does \emph{not} necessarily hold, because some
$j$-sets which once were active may already be explored.

For every $t$, let $R_t$ be the set of all discovered $j$-sets at time $t$ which were new starts,
and let $S_t$ be the discovered $j$-sets up to time $t$ which are standard.
Thus, for all $t$,
\[ R_t \cup S_t = A_t \cup E_t. \]
Note that if the query at time $t$ is answered positively,
then $|S_t| = |S_{t-1}| + \binom{k-j}{a}$, and otherwise $|S_t| = |S_{t-1}|$.
Thus, if $X_1, X_2,\ldots $ are the indicator variables that track which queries
are answered positively,
i.e.\ $X_i$ is $1$ if the $i$-th $k$-tuple queried forms an edge and $0$ otherwise,
then we have
\begin{align}
	|S_t| = \binom{k-j}{a} \sum_{i=1}^t X_i. \label{eqn:discoveredviaqueries}
\end{align}
Note that with input hypergraph $H=H^k(n,p)$, the $X_1,X_2,\ldots$ are simply 
i.i.d.\ Bernoulli random variables with probability $p$.
In particular, using Chernoff bounds, we can approximate $|S_t|$ when $t$ is large.
For completeness, the proof is given in Appendix~\ref{app:algprops}.

\begin{proposition} \label{proposition:usingchernoff}
	Let $p=\frac{1+\eps}{\binom{k-j}{a}\binom{n-j}{k-j}}$, let $t=t(n)\in \NN$, and let $0\le \gamma=\gamma(n)=O(1)$.
	Then when \pathfinder\ is run with input $k,j$ and $H=H^k(n,p)$, with probability
	at least $1-\exp(-\Theta(\gamma^2 pt))$ we have
	\[ (1 - \gamma)\frac{(1+\eps)t}{\binom{n-j}{k-j}} \leq |S_t| \leq (1 + \gamma)\frac{(1+\eps)t}{\binom{n-j}{k-j}}. \]
	In particular, if $\gamma^2 t n^{-(k-j)} \to \infty$, then these inequalities hold \whp.
\end{proposition}

Note that this proposition gives a lower bound on the number of discovered
$j$-sets, but it does not immediately give an upper bound, since it says
nothing about the number of new starts that have been made.
(Later the number of new starts will be bounded
by Proposition~\ref{prop:newstartsbound} in the case $j\ge 2$, ; we will not need such
an upper bound in the case $j=1$.)

How many queries are made from a given $j$-set $J$ before it is declared explored?
Clearly $\binom{n-j}{k-j}$ is an upper bound, since this is the number 
of $k$-sets that contain $J$, but some of these are excluded in the algorithm,
and we will need a lower bound.
In what follows, for convenience we slightly abuse terminology by referring to querying 
not a $k$-set $K\supset J$, but rather the $(k-j)$-set $K\setminus J$.
(If $J$ is already determined, this is clearly equivalent.)

There are two reasons why a $(k-j)$-set disjoint from the current $j$-set $J$
may never be queried---
either it contains a vertex of the current path, or it contains an explored $j$-set.

\begin{definition} \label{definition:forbidden}
Consider an exploration of a $k$-uniform hypergraph $H$ using \pathfinder.
	Given $t$, let $J$ be the last active set in the stack of $A_t$.
	We call a $(k-j)$-set $X\subset V(H)\setminus J$ \emph{forbidden at time $t$}, if
	\begin{enumerate}[label=\normalfont{(\arabic*)}]
		\item $X\cap V(P_t)\neq\emptyset$, or \label{item:forbidden-path}
		\item there exists an explored $j$-set $J' \in E_t$ such that $J' \subset (J\cup X)$. \label{item:forbidden-explored}
	\end{enumerate}
	If $X$ satisfies~\ref{item:forbidden-path}
	we say $X$ is a \emph{forbidden set of type 1}; if it satisfies~\ref{item:forbidden-explored}
	we say it is a \emph{forbidden set of type 2}.
	Let $F^{(1)}=F^{(1)}_t$ and $F^{(2)}=F^{(2)}_t$ denote
	the corresponding sets of forbidden $(k-j)$-sets at time $t$, and 
	let $F=F_t:=F^{(1)}_t \cup F^{(2)}_t$ be the set of all forbidden $(k-j)$-sets at time $t$.
\end{definition}
Observe that a $(k-j)$-set might be a forbidden set of both types, i.e.\ may lie in both $F^{(1)}$ and $F^{(2)}$.
The following consequence of the definition of forbidden $(k-j)$-sets is crucial:
if $J$ is declared explored at time $t$ and a $(k-j)$-set $X$ disjoint from $J$
is not in $F_t$, then $X$ was queried from $J$ by the algorithm (at some time $t' \leq t$).
Thus, if the number of forbidden sets at time $t$ is ``small'',
then a ``large'' number of queries were required to declare $J$ explored.

Our aim is to bound the size of $F_t = F^{(1)}_t \cup F^{(2)}_t$.
If the \pathfinder\ algorithm has not found a long path,
then $F^{(1)}_t$ is small.
More precisely, we obtain the following bound.

\begin{proposition}\label{prop:boundforbidden1length}
	For all times $t \ge 0$,
	\[ |F^{(1)}_t| \leq \ell_t\cdot (k-j) \binom{n-j-1}{k-j-1}. \]
\end{proposition}

\begin{proof}
	Let $J$ be the current active $j$-set in $A_t$.
	A $(k-j)$-set $X$ is in $F^{(1)}_t$ if and only if $X \cap J = \emptyset$ and $X \cap V(P_t) \neq \emptyset$;
	thus $|F^{(1)}_t| \leq |V(P_t) \setminus J| \binom{n-j-1}{k-j-1}$.
	Since $J \subset V(P_t)$ and $P_t$ has $\ell_t$ edges,
	we have $|V(P_t) \setminus J| = \ell_t\cdot (k-j)$, and the desired bound follows.
\end{proof}

It remains to estimate the number of forbidden sets of type 2.
To achieve this, in the next section we will give more precise estimates
on the evolution of the algorithm run with input $H^k(n,p)$
(and in particular the evolution of discovered $j$-sets,
which certainly includes all explored $j$-sets).

We will need to treat the case $j=1$ separately from the case $j\ge 2$.
We begin with the case $j=1$, since this is significantly easier
but introduces some of the ideas that will be used in the more complex
case $j\ge 2$.

\section{Algorithm analysis: loose case ($j=1$)} \label{sec:loose}

The case $j=1$ is different from all other cases because the $j$-sets of the
exploration process are simply vertices. This is important because there
is a certain interplay between $j$-sets and vertices regarding where a path
``lies''---in general, $j$-sets can only be blocked because they were previously explored, but
vertices can be blocked because they are in the current path. Furthermore, for $j\ge 2$,
we may revisit some vertices from a discarded branch of the
depth-first search process, but for $j=1$, since $j$-sets and vertices are the same,
this is no longer possible.

This fundamental difference is reflected in the fact that the length of the longest
path discovered by the \pathfinder\ algorithm in the supercritical case
is significantly shorter for $j=1$ (i.e.\ $\Theta(\eps^2n)$ rather than $\Theta(\eps n)$).
Indeed, it seems likely that this is in fact best possible up to a constant factor,
i.e.\ that the longest loose path has length $\Theta(\eps^2 n)$,
rather than that either the algorithm or our analysis is far too weak.
This is certainly the case for graphs, i.e.\ for $k=2$; we will
discuss this for general $k$ in more detail in Section~\ref{sec:concluding}.

For convenience, we restate the result we are aiming to prove as a lemma.

\begin{lemma}\label{lem:loosesupercriticallower}
Let $k\in \NN$ and let $\eps = \eps(n)$ satisfy $\eps^3 n \xrightarrow{n\to \infty} \infty$. Let
$$
p=(1+\eps)p_0 = \frac{1+\eps}{(k-1)\binom{n-1}{k-1}}
$$
and let $L$ be the length of the longest loose path in $H^k(n,p)$.
Then for all $\delta \gg \eps$ satisfying $\delta^2\eps^3 n \xrightarrow{n\to \infty}\infty$, \whp
		 \[ L\ge (1 - \delta)\frac{\eps^2 n}{4(k-1)^2} .\]
\end{lemma}

We define
$$
\ell_0 := \frac{(1-\delta)\eps^2 n}{4(k-1)^2},
$$
so our goal is to show that \whp\ the \pathfinder\ algorithm
discovers a path of length at least $\ell_0$.
We also define
$$
T_0:= \frac{\eps n \binom{n-1}{k-1}}{2(k-1)} = \frac{\eps n}{2(k-1)^2 p_0}.
$$
We will show that \whp\ at some time $t\le T_0$, we have $\ell_t \ge \ell_0$,
as required.
We begin with the following proposition, which is a simple application of Proposition~\ref{proposition:usingchernoff}.
For completeness, the proof appears in Appendix~\ref{app:algprops}.

\begin{proposition}\label{prop:loosechernoff}
At time $t=T_0$, \whp\ we have
\[
	|A_t \cup E_t|\ge (1-o(\delta \eps))(k-1)pt.
\]
\end{proposition}

Let $T_1$ denote the first time $t$ at which
\begin{align}
|A_t\cup E_t| = \left(1-\frac{\delta\eps}{3}\right)(k-1)pT_0 = \left(1-\frac{\delta\eps}{3}\right)(1+\eps)\frac{\eps n}{2(k-1)}\label{eq:disct1}
\end{align}
(recall that we ignore floors and ceilings). Then from Proposition~\ref{prop:loosechernoff},
we immediately obtain the following.
\begin{corollary}\label{cor:loosechernoff}
\Whp\ $T_1\le T_0$.
\end{corollary}

We claim furthermore that this inequality implies that we must have
a long loose path.

\begin{proposition}\label{prop:loosefinalcontradiction}
If $T_1\le T_0$, then at time $t=T_1$ we have $\ell_t \ge \ell_0$.
\end{proposition}

\begin{proof}
Suppose for a contradiction that $T_1\le T_0$, but that at time $t=T_1$ we have $\ell_t < \ell_0$.
Then by~\eqref{eq:disct1} and~\eqref{eqn:activeviapath}
\begin{align*}
|E_t| & = |A_t\cup E_t| -|A_t|\\
 & \ge \left(1-\frac{\delta\eps}{3}\right)(k-1)pT_0 - ((k-1)\ell_0 +1) \\
& = \left(1+\eps-\frac{\delta\eps}{3} - O(\delta\eps^2)\right)(k-1)p_0 T_0 - \frac{(1-\delta)\eps^2 n}{4(k-1)} - o(\delta\eps^2 n) \\
& = \left(1+\eps-\frac{\delta\eps}{3}-\frac{(1-\delta)\eps}{2}-O(\delta\eps^2) - o(\delta\eps)\right)(k-1)p_0 T_0 \\
& \ge \left(1+ \frac{\eps}{2} + \frac{\delta\eps}{7} \right)(k-1) p_0 T_0,
\end{align*}
where we have used the fact that $(k-1)p_0 T_0 = \frac{\eps n}{2(k-1)} = \Theta(\eps n)$, 
and that $\delta\eps^2 n \ge \eps^3 n \to \infty$.

On the other hand, $N_t$, the set of neutral vertices, satisfies 
\begin{align*}
|N_t| = n-|A_t \cup E_t| & \stackrel{\eqref{eq:disct1}}{=} n-(1-o(\delta\eps))(1+\eps)\frac{\eps n}{2(k-1)} \\
& = \left(1-\frac{\eps}{2(k-1)} + o(\delta\eps) + O(\eps^2)\right)n.
\end{align*}
Note that no vertex of $N_t$ can possibly have been forbidden at any time $t' \leq t$.
This implies, since the vertices of $E_t$ are fully explored, that from each explored vertex we certainly queried any
$k$-set containing the vertex and $k-1$ vertices of $N_t$.
Thus the number of queries $t$ that we have made so far certainly satisfies
\begin{align*}
t & \ge |E_t| \binom{|N_t|}{k-1}\\
& \ge \left(1+\frac{\eps}{2} + \frac{\delta\eps}{7} \right)(k-1)p_0 T_0
\cdot
\left(1+O\left(\frac{1}{n}\right)\right)\frac{\left(1-\frac{\eps}{2(k-1)} + o(\delta\eps) + O(\eps^2)\right)^{k-1}n^{k-1}}{(k-1)!}\\
& = \left(1+\frac{\eps}{2} + \frac{\delta\eps}{7}\right)T_0 \cdot
\left(1+O\left(\frac{1}{n}\right)\right)\left(1-\frac{\eps}{2} + o(\delta\eps) +O(\eps^2)\right) \\
& = \left(1+\frac{\delta \eps}{7} + o(\delta\eps) + O(\eps^2)\right)T_0 \\
& > T_0,
\end{align*}
which gives the required contradiction since we assumed that $t= T_1 \le T_0$.
\end{proof}

\begin{proof}[Proof of Lemma~\ref{lem:loosesupercriticallower}]
The statement of Lemma~\ref{lem:loosesupercriticallower}
follows directly from
Corollary~\ref{cor:loosechernoff} and Proposition~\ref{prop:loosefinalcontradiction}.
\end{proof}

Let us note that although we proved that \whp\ $\ell_t \ge \ell_0$
at \emph{some} time $t\le T_0$, with a small amount of extra work we could
actually prove that this even holds at exactly $t=T_0$: we would need a
corresponding upper bound in Proposition~\ref{prop:loosechernoff},
which follows from a Chernoff bound on the number of edges
discovered so far and an upper bound on the number of new starts we have made
by time $T_0$.

\section{Algorithm analysis: high-order case ($j\ge 2$)}\label{sec:alganalysis}

In the case $j\ge 2$, 
we will use the \pathfinder\ algorithm to study $j$-tight paths in $H^k(n,p)$
by running the algorithm up to a certain \emph{stopping time} $\Tstop$, i.e.\ until
we have made $\Tstop$ queries. In order to define $\Tstop$, we need some additional definitions.

Given some time $t\ge 0$ let $D_t$ denote the set of all $j$-sets which are discovered by time $t$.
With a slight abuse of notation, we will sometimes also use $D_t$ to denote the
$j$-uniform hypergraph on vertex set $[n]$ with edge set $D_t$.
Note that a $j$-set $J$ lies in $D_t$ if and only if there exists $t' \leq t$
such that $J \in A_{t'}$, or in other words, every $j$-set which is discovered at time $t$ was active
at some time $t'\le t$.
Also, note that for every $t_1 \leq t_2$, $D_{t_1} \subseteq D_{t_2}$, i.e.\
the sequence of discovered $j$-sets is always increasing (although the sequence of active sets $A_t$ is not).

Suppose that $0\le i \le j$ and that $I$ is an $i$-set. 
Then define $d(I)=d_t(I)=\deg_{D_t}(I)$ to be the number of $j$-sets of $D_t$ that contain $I$.

\begin{definition}
Let $\eps \ll \delta \le 1$ be as in Theorem~\ref{thm:mainresult}\ref{item:thm-supercritical},\footnote{Recall
from Remark~\ref{rem:weaker} that we will not actually use the additional condition $\delta \gg \frac{\ln n}{\eps^2 n}$
for the proof of the lower bound, c.f.\ Lemma~\ref{lem:mainlemmastoppingtime}.}
and recall that $|R_t|$ is the number of new starts made by time $t$.
 Let
 $$
 C_{k,j,j-1}\gg C_{k,j,j-2}\gg \dotsb \ge C_{k,j,0} \gg 1
 $$
 be some sufficiently large constants
 and let $0<\beta \ll 1$ be a sufficiently small constant.
 Define
 $$
 T_0:=\frac{n^{k-j+1}}{\eps}.
 $$
We define $\Tstop$ to be the smallest time $t$ such that one of the following stopping conditions hold:
\begin{enumerate}[label=\textnormal{\textbf{(S\arabic*)}}]
	\item\label{stoplength} \pathfinder\ found a path of length
	at least $(1-\delta)\frac{\eps n}{(k-j)^2}$;
	\item\label{stoptime} $t=T_0$;
	\item\label{stopnewstarts} $|R_t|\geq 2(k-j)!\sqrt{\frac{tn^\beta}{n^{k-j}}} +\frac{n^\beta}{2}$;
	\item\label{stopdegree} There exists some $0\le i \le j-1$ and an $i$-set $I$ with
	$d_t(I)\geq\frac{C_{k,j,i}t}{n^{k-j+i}}+n^{\beta}$.
\end{enumerate}
\end{definition}

We first observe that $\Tstop$ is well-defined.

\begin{claim}
If \pathfinder\ is run on inputs $k,j$ and any $k$-uniform hypergraph $H$ on $[n]$,
then one of the four stopping conditions is always applied.
\end{claim}

\begin{proof}
If none of the stopping conditions is applied, the algorithm will
continue until all $j$-sets are explored
(since a new start is always possible from any neutral $j$-set).
If this occurs at time $t\ge T_0$, then~\ref{stoptime} would already
have been applied (if none of the other stopping conditions were
applied first). On the other hand,
if this occurs at time $t\le T_0$, then~\ref{stopdegree} is certainly
satisfied with $i=0$ and $I=\emptyset$.
\end{proof}

We will often use the fact that for $t\le \Tstop$, the (non-strict) inequalites
in stopping conditions~\ref{stoplength},~\ref{stopnewstarts} and~\ref{stopdegree} are reversed.
For example, for $t\le \Tstop$ we have $|R_t|\le 2(k-j)!\sqrt{\frac{tn^\beta}{n^{k-j}}} +n^\beta$.
This is because
$$|R_t|\le |R_{t-1}|+1 < 2(k-j)!\sqrt{\frac{(t-1)n^\beta}{n^{k-j}}} +n^\beta +1,$$
where the second inequality holds because we did not apply~\ref{stopnewstarts} by time $t-1$
(and recall that we ignore floors and ceilings). In such a situation, we will slightly abuse terminology
by saying that ``by~\ref{stopnewstarts}'' we have $|R_t|\le 2(k-j)!\sqrt{\frac{tn^\beta}{n^{k-j}}} +n^\beta$.

Our main goal is to show that \whp\ it is~\ref{stoplength} which is applied first,
i.e.\ the algorithm has indeed discovered a path of the appropriate length.

\begin{lemma}\label{lem:mainlemmastoppingtime}
	Let $k,j\in\NN$ satisfy $2 \leq j \leq k-1$. Let $a\in\NN$ be the unique integer satisfying
	$1 \leq a \leq k-j$ and $a \equiv k \bmod (k-j)$.
	Let $\eps = \eps(n) \ll 1$ satisfy $\eps^3 n \xrightarrow{n\to\infty} \infty$
	and let
	$$
	p_0 = p_0(n;k,j) := \frac{1}{\binom{k-j}{a} \binom{n-j}{k-j}}.
	$$
	Let $L$ be the length of the longest $j$-tight path in $H^k(n,p)$,
	and let $\delta \gg \eps$.

	Suppose \pathfinder\ is run with input $k,j$ and $H=H^k(n,p)$. Then \whp\ \ref{stoplength} is applied.
	In particular, \whp\
	$$L \ge \ell_{\Tstop} =(1-\delta)\frac{\eps n}{(k-j)^2}.$$
\end{lemma}

For the rest of this section, we will assume that all parameters are as
defined in Lemma~\ref{lem:mainlemmastoppingtime}.

We first prove an auxiliary lemma which gives an upper bound on the number of forbidden $(k-j)$-sets
up to time $\Tstop$.
Recall that $F^{(1)}_t$ and $F^{(2)}_t$ denote the sets of forbidden $(k-j)$-sets
at time $t$ of types $1$ and $2$, respectively.
Let $f^{(i)} = f^{(i)}_t:=|F^{(i)}_t|$ for $i=1,2$.

\begin{lemma}\label{lem:forbiddensets}
	Let $t\le \Tstop$.
	Then
	\[f^{(1)}+f^{(2)}\leq (1-\delta/2)\eps\binom{n-j}{k-j}\;.\]
	In particular, 
	from every explored $j$-set we made at least 
	\[ (1-\eps +\delta\eps/2)\binom{n-j}{k-j}\]
	queries.
\end{lemma}

\begin{proof}
	Due to condition~\ref{stoplength}, the length $\ell_{t}$ of the path $P_{t}$ at any time $t$
	is at most $\frac{(1-\delta)\eps n}{(k-j)^2}$.
	Thus by Proposition~\ref{prop:boundforbidden1length} we have that
	\begin{equation}\label{eq:f1bound}
	f^{(1)}\leq \frac{(1-\delta)\eps n}{(k-j)}\cdot\binom{n-j-1}{k-j-1}
	\leq\left(1-\frac{2\delta}{3}\right)\eps\binom{n-j}{k-j}\;.
	\end{equation}
	By condition~\ref{stoptime}, we have $\Tstop\leq \frac{n^{k-j+1}}{\eps}$. 
	Furthermore, by condition~\ref{stopdegree}, for any $0\leq i\leq j-1$ and any $i$-set $I$ we have
	$$d_{t}(I)\leq d_{\Tstop}(I)
	\le \frac{C_{k,j,i}\Tstop}{n^{k-j+i}} + n^\beta
	\leq \frac{C_{k,j,i}}{\eps n^{i-1}} +n^\beta.$$
	Observe that if $J$ is the current $j$-set, any forbidden $(k-j)$-set of type 2 can be identified by:
	\begin{itemize}
	\item choosing an integer $i=0,\ldots,j-1$;
	\item choosing a proper subset $I\subset J$ of size $i$ (there are $\binom{j}{i}$ possibilities);
	\item choosing an explored
	(and therefore discovered) $j$-set $J'\supset I$ such that $(J'\setminus I)\cap J = \emptyset$,
	(at most $d_t(I)$ possibilities);
	\item choosing a $k$-set $K$ containing both $J$ and $J'$ (there are $\binom{n-2j+i}{k-2j+i}$ possibilities).
	\end{itemize}
	Then the forbidden $(k-j)$-set is $K\setminus J$.
	Note that if $j>k/2$, then $k-2j+i$ may be negative for some values of $i$.
	In this case we interpret $\binom{n-2j+i}{k-2j+i}$ to be zero.
	
	Therefore  we obtain	
	\begin{align*}
	f^{(2)}&\leq \sum_{i=0}^{j-1}\binom{j}{i}\cdot \left(\max_{|I|=i} d_{t}(I)\right)\cdot\binom{n-2j+i}{k-2j+i}\\
	&\leq \sum_{i=0}^{j-1}2^j\cdot \left(\frac{C_{k,j,i}}{\eps n^{i-1}} +n^\beta\right)
		\cdot O(n^{-j+i}) \binom{n-j}{k-j}\\
	& = O\left(\frac{1}{\delta\eps^2 n^{j-1}} + \frac{n^{\beta}}{\delta \eps n} \right) \delta \eps \binom{n-j}{k-j}.
	\end{align*}
	Now recall that $\delta \gg \eps$ and that we are considering the case $j\ge 2$,
	which means that $\delta\eps^2 n^{j-1} \ge \eps^3 n \to \infty$.
	Furthermore $\beta \ll 1$, which implies that $\delta\eps n^{1-\beta} \ge \eps^2 n^{2/3}\to \infty$,
	so we obtain
	$$
	f^{(2)} = o(1)\delta\eps\binom{n-j}{k-j}.
	$$ 
	Together with~\eqref{eq:f1bound}, this leads to
	\[f^{(1)}+f^{(2)}\leq \left(1-\frac{2\delta}{3}+o(\delta)\right)\eps\binom{n-j}{k-j} \le (1-\delta/2)\eps\binom{n-j}{k-j} \]
	as claimed.
\end{proof}

Our aim now is to prove Lemma \ref{lem:mainlemmastoppingtime},
i.e.\ that \whp\ stopping condition~\ref{stoplength} is applied. Our strategy
is to show that \whp\ each of the other three stopping conditions is \emph{not} applied.
The arguments for~\ref{stoptime} and~\ref{stopnewstarts} are almost identical,
so it is convenient to handle them together.
We begin with the following proposition.

\begin{proposition}\label{prop:newstarts:probabilistic}
There exists an event $\cA$ such that:
\begin{enumerate}[label=\normalfont{(\roman*)}]
\item $\Pr(\cA)= 1-o(1);$
\item if $\cA$ holds and either~\ref{stoptime} or~\ref{stopnewstarts} is applied
at time $t= \Tstop$, then
$$
|E_t| \geq \frac{(1-2\delta\eps/5)(1+\eps)t}{\binom{n-j}{k-j}}.$$
\end{enumerate}
\end{proposition}

\begin{proof}
We first define the event $\cA$ explicitly.
For any time $t>0$ we define
$$
\gamma_t:= \begin{cases}
\vspace{0.1cm}
\sqrt{\frac{n^{k-j+\beta}}{t}} & \mbox{if } t< T_0, \\
\frac{\delta \eps}{3} & \mbox{otherwise,}
\end{cases}
$$
and let
$$
\cA_t := \left\{|S_t| \ge \left(1-\gamma_t\right)\frac{(1+\eps)t}{\binom{n-j}{k-j}}\right\}.
$$
Now we define
$$
\cA := \bigcap_{\frac{n^{k-j+\beta}}{4(k-j)!} \le t \le T_0} \cA_t.
$$
We now need to show that the two properties of the proposition are satisfied for this choice of $\cA$.
First observe that for $\frac{n^{k-j+\beta}}{4(k-j)!} \le t < T_0$,
Proposition~\ref{proposition:usingchernoff} (applied with $\gamma=\gamma_t$) implies that
\[\Pr(\cA_t)  \ge 1-\exp\left(-\Theta\left(\gamma_t^2 pt\right)\right)
 \ge 1-\exp\left(-\Theta\left(\gamma_t^2 \frac{t}{n^{k-j}}\right)\right)
\ge 1-\exp\left(-\Theta\left(n^\beta\right)\right).\]

On the other hand, for $t= T_0$ again Proposition~\ref{proposition:usingchernoff}
implies that
$$
\Pr(\cA_{T_0})\ge 1-\exp\left(-\Theta\left(\gamma_t^2 pT_0\right)\right)
= 1-\exp\left(-\Theta\left(\delta^2\eps n\right)\right) =1-o(1),
$$
where the convergence holds because $\delta^2\eps n \ge \eps^3 n \to \infty$.
Therefore by applying a union bound,
$$
\Pr(\cA) \ge 1-T_0 \exp\left(-\Theta\left(n^\beta\right)\right) - o(1) = 1-o(1),
$$
as required.

We now aim to prove the second statement, so let us assume that $\cA$ holds,
and 
we make a case distinction according to which of~\ref{stopnewstarts}
and~\ref{stoptime} is applied.

\subsubsection*{Case 1: \ref{stopnewstarts} is applied}

By applying Lemma \ref{lem:forbiddensets} we can bound the number of queries
made from each explored $j$-set at any time $t\le \Tstop$ from below by
\[(1-\eps+\delta\eps/2)\binom{n-j}{k-j} \ge \frac{3n^{k-j}}{4(k-j)!}.\]
In particular, since~\ref{stopnewstarts} is applied,
we must have made at least $n^\beta/2$ new starts, and therefore
at least $n^\beta/2 -1 \ge n^\beta/3$ many $j$-sets are explored.
Thus we have made at least $\frac{n^\beta}{3} \cdot \frac{3n^{k-j}}{4(k-j)!}$
queries, and therefore we may assume that $\Tstop \ge \frac{n^{k-j+\beta}}{4(k-j)!}$.
(Note that this in particular motivates why the definition of $\cA$
did not include any $\cA_t$ for $t< \frac{n^{k-j+\beta}}{4(k-j)!}$.)

Furthermore, since~\ref{stoptime} is \emph{not} applied, we have $\Tstop<T_0$.
Therefore, the fact that $\cA$ holds tells us that for $t = \Tstop$,

\begin{equation}\label{eq:intermediatetpartial}
|D_t| \ge |S_t| + |R_t| \ge \left(1-\gamma_t\right)(1+\eps)\frac{t}{\binom{n-j}{k-j}} + |R_t|.
\end{equation}
Since~\ref{stopnewstarts} is applied at $t = \Tstop$,
we further have
$$
|R_t|\ge 2(k-j)!\sqrt{\frac{t n^\beta}{n^{k-j}}} \ge \frac{3\gamma_t t}{2\binom{n-j}{k-j}}.
$$
Substituting this inequality into~\eqref{eq:intermediatetpartial}, we obtain 
$$
|D_t| \ge \left(1-\gamma_t\right)(1+\eps)\frac{t}{\binom{n-j}{k-j}} + \frac{3\gamma_t t}{2\binom{n-j}{k-j}}
\ge (1+\eps)\frac{t}{\binom{n-j}{k-j}}.
$$

Furthermore, since~\ref{stopnewstarts} is applied at $t = \Tstop$,
a new start must have been made at time $t$.
This implies that the set of active sets at time $A_t$ was empty, i.e. $|A_t| = 0$.
This means that
\begin{equation*}
|E_t| = |D_t| \ge (1+\eps)\frac{t}{\binom{n-j}{k-j}}
\ge \frac{(1-2\delta\eps/5)(1+\eps)t}{\binom{n-j}{k-j}},
\end{equation*}
as claimed.

\subsubsection*{Case 2: \ref{stoptime} is applied}

We will use
the trivial bound $|R_t|\ge 0$,
and therefore $\cA$ tells us that at time $t = T_0 = \Tstop$ we have
\begin{equation*}
|D_t| = |S_t|+|R_t| \ge \left(1-\frac{\delta\eps}{3}\right)(1+\eps)\frac{t}{\binom{n-j}{k-j}}.
\end{equation*}
Furthermore, by~\ref{stoplength},
$$\ell_t= O(\eps n),$$ and therefore
by~\eqref{eqn:activeviapath}
\begin{align*}
|A_{t}| \leq 1 + \binom{k-j}{a} \ell_t & = O(\eps n)  = O\left(\frac{\eps^2 T_0}{n^{k-j}}\right).
\end{align*}
Thus 
the number of explored sets at time $T_0$ satisfies 

\begin{equation*}\label{eq:T0final}
|E_{T_0}| = |D_{T_0}| - |A_{T_0}|
\geq \frac{\left((1-\delta\eps/3)(1+\eps)-O\left(\eps^2 \right)\right)T_0}{\binom{n-j}{k-j}}  
\geq \frac{(1-2\delta\eps/5)(1+\eps)T_0}{\binom{n-j}{k-j}},
\end{equation*}
where in the last step we have used the fact that $\delta \gg \eps$.
\end{proof}

The previous result enables us to prove the following.

\begin{proposition}\label{prop:newstartsbound}\label{SCnewstarts}\label{prop:newstartsboundj1}
	\Whp\ neither~\ref{stoptime} nor~\ref{stopnewstarts} is applied.
\end{proposition}
\begin{proof}
For any time $t\ge 0$, let us define the event
$$
\cE_t:= \left\{|E_t| \geq \frac{(1-2\delta\eps/5)(1+\eps)t}{\binom{n-j}{k-j}}\right\},
$$
i.e.\ that the bound on $|E_t|$ from Proposition~\ref{prop:newstarts:probabilistic} holds.
We will show that in fact it is not possible that
$\cE_t$ holds for any $t\le \Tstop$.
Therefore, Proposition~\ref{prop:newstarts:probabilistic}
implies that the probability that
one of~\ref{stopnewstarts} and~\ref{stoptime} is applied is at most $1-\Pr(\cA)= o(1)$.
So suppose for
a contradiction that $\cE_t$ holds for some $t\le \Tstop$.

As in the proof of Proposition~\ref{prop:newstarts:probabilistic},
an application of Lemma \ref{lem:forbiddensets} implies that from each explored $j$-set at any time $t\le \Tstop$ 
we made at least
\[(1-\eps+\delta\eps/2)\binom{n-j}{k-j} \ge \frac{3n^{k-j}}{4(k-j)!}\]
queries.
Therefore, by Proposition~\ref{prop:newstarts:probabilistic},
the total number $t$ of queries made satisfies

\[t  \geq |E_{t}| \cdot(1-\eps+\delta \eps/2)\binom{n-j}{k-j}
\geq (1-2\delta\eps/5+\delta\eps/2+O(\eps^2))t
>t,\]
yielding the desired contradiction.
\end{proof}

We next prove that \whp\ \ref{stopdegree} is not applied.
This may be seen as a form of \emph{bounded degree lemma}.
Both the result and the proof are inspired by similar results in~\cite{CooleyKangKoch18,CooleyKangPerson18}.

The intuition behind this stopping condition is that the \emph{average} degree of an $i$-set
should be of order $\frac{tp}{n^i} \sim \frac{t}{n^{k-j+i}}$,
and~\ref{stopdegree} guarantees that, for $t\le \Tstop$, no $i$-set exceeds this by more than
a constant factor.
The $n^\beta$-term can be interpreted as an error term which takes over
when the average $i$-degree (i.e.\ the average degree over all $i$-sets) is too small to guarantee an appropriate
concentration result.

Note, however, that due to the choice of $T_0$, the average $i$-degree
is actually much smaller than $n^\beta$ for any $i\ge 2$ (and possibly
even for $i=1$ if $\eps=\Omega(n^{-\beta})$). Meanwhile, the statement
for $i=0$ is simply a statement about the number of discovered $j$-sets,
which follows from a simple Chernoff bound on the number of edges discovered,
together with~\ref{stopnewstarts} to bound the number of new starts.
Thus the strongest and most interesting case of the statement is when
$i=1$; 
nevertheless, our proof strategy is strong enough to cover all $i$
and would even work for any $t>T_0$, provided~\ref{stopnewstarts}
has not yet been applied.

\begin{lemma}\label{lem:SCdegrees}
	\Whp\ \ref{stopdegree} is not applied.
\end{lemma}

\begin{proof}
We will prove that the probability that~\ref{stopdegree} is applied
at a particular time $t\le \Tstop$, i.e.\ before any other stopping condition
has been applied, is at most $\exp\left(-\Theta\left(n^{\beta/2}\right)\right)=o(n^{-k})$,
and then a union bound over all possible $t$ completes the argument.

	We will prove the lemma by induction on $i$. For $i=0$ the statement is
	just that the number of discovered $j$-sets is at most $C_{k,j,0}t / n^{k - j} +n^\beta$,
	which follows from Lemma~\ref{lem:Chernoffwitherror} and~\ref{stopnewstarts}.
	More precisely, using~\eqref{eqn:discoveredviaqueries}
	and applying Lemma~\ref{lem:Chernoffwitherror} with $\alpha=\beta/2$, we have that
	$$
	\Pr\left(\frac{|S_t|}{\binom{k-j}{a}}\ge 2tp+n^{\beta/2}\right)
	\le \exp\left(-\Theta\left(n^{\beta/2}\right)\right).
	$$
	Furthermore, by~\ref{stopnewstarts}, we have
	\begin{align*}
	|R_t| & \le 2(k-j)!\sqrt{\frac{t n^\beta}{n^{k-j}}} + \frac{n^\beta}{2} \\
	& \le \begin{cases}
	\frac{3n^\beta}{4} & \mbox{if } t\le \frac{n^{k-j+\beta}}{64((k-j)!)^2},\\
	16((k-j)!)^2 \frac{t}{n^{k-j}} + \frac{n^\beta}{2} & \mbox{if } t\ge \frac{n^{k-j+\beta}}{64((k-j)!)^2}
	\end{cases}\\
	& \le 16((k-j)!)^2 \frac{t}{n^{k-j}} + \frac{3n^\beta}{4}.
	\end{align*}
	Thus with probability at least $1-\exp\left(-\Theta\left(n^{\beta/2}\right)\right)$
	we have
	\begin{align*}
	|D_t|=|S_t|+|R_t| & \le \binom{k-j}{a}\left(2tp + n^{\beta/2}\right) + 16((k-j)!)^2 \frac{t}{n^{k-j}} + \frac{3n^\beta}{4} \\
	& \le \left(3(k-j)! + 16((k-j)!)^2\right) \cdot \frac{t}{n^{k-j}} + n^\beta \\
	& \le \frac{20((k-j)!)^2 t}{n^{k-j}} + n^\beta,
	\end{align*}
	and since we chose $C_{k,j,0}\gg 1$, and in particular $C_{k,j,0}> 20((k-j)!)^2$,
	this shows that \whp~\ref{stopdegree}
	is not applied because of $I=\emptyset$ (i.e.\ with $i=0$).
	So we will assume that $i \ge 1$ and that~\ref{stopdegree} is not applied for $0,1,\ldots,i-1$.

	Given $1 \le i \le j-1$ and an $i$-set $I$, let us consider the possible
	ways in which some $j$-sets containing $I$ may become active.
	
	\begin{itemize}
		\item A \emph{new start} at $I$ occurs when there are no active $j$-sets
		and we make a new start at a
		$j$-set which happens to contain $I$. In this case $d(I)$ increases by $1$;
		\item A \emph{jump} to $I$ occurs when we query a $k$-set containing $I$ from a $j$-set
		not containing $I$ and discover an edge.
		In this case $d(I)$
		increases by at most $\binom{k-j}{a}$ (the number of new $j$-sets which become active in a batch,
		each of which may or may not contain $I$);
		\item A \emph{pivot} at $I$ occurs when we query a $k$-set from a $j$-set containing $I$
		and discover an edge. In this case
		$d(I)$ increases by at most $\binom{k-j}{a}$.
	\end{itemize}
	
	Each possibility makes a contribution to the degree of $I$ according to
	how many $j$-sets containing $I$ become active as a result of each type
	of event.
	We bound the three contributions separately.
	
	\textbf{New starts:} Whenever we make a new start, we choose the starting $j$-set
	according to some (previously fixed) random ordering $\sigma_j$ (recall
	that $\sigma_j$ was a permutation of the $j$-sets chosen uniformly at random
	during the initialisation of the algorithm).
	By~\ref{stopnewstarts}, at time $t\le \Tstop$ the number of 
	new starts we have made is
	$$
	|R_t|\le 2(k-j)!\sqrt{\frac{tn^{\beta}}{n^{k-j}}}+\frac{n^\beta}{2}.
	$$
	Observe that
	$$
	\sqrt{\frac{tn^\beta}{n^{k-j}}} \le \begin{cases}
	n^\beta & \mbox{if } t\le n^{k-j+\beta},\\
	\frac{t}{n^{k-j}} & \mbox{if } t\ge n^{k-j+\beta},
	\end{cases}
	$$
	which means that the number of new starts satisfies
	$$
	|R_t|
	\le 2(k-j)!tn^{j-k}+3(k-j)!n^\beta =:N^*.
	$$
	Since the new starts are distributed randomly,
	the probability that a $j$-set chosen for a new start at time $t'\le t$ contains
	$I$ is precisely the proportion of neutral $j$-sets at time $t'$ which contain
	$I$. Since~\ref{stopdegree} has not yet been applied, in particular with $i=0$,
	the total number of non-neutral $j$-sets (which cannot be chosen for a new start)
	at time $t'\le t$ is at most
	\begin{align*}
	d_{t'}(\emptyset) \le d_t(\emptyset) \le \frac{C_{k,j,0}t}{n^{k-j}} + n^\beta \le \frac{C_{k,j,0}n}{\eps} + n^\beta 
	\le n^{4/3} = o(n^{j}).
	\end{align*}
	Thus the probability that the $j$-set chosen contains $I$ is at most
	$$
	\frac{\binom{n-i}{j-i}}{\binom{n}{j}-o(n^j)} \le 2j!n^{-i}.
	$$
	Therefore
	the number of new starts containing $I$ is dominated
	by $\Bi(N^*,2j!n^{-i})$, which has expectation
	at most $4k! t n^{j-k-i}+1$ (since $n^{\beta-i}=o(1)$).
	By Lemma~\ref{lem:Chernoffwitherror},
	with probability at least $1-\exp (-\Theta(n^{\beta/2}))$ the number of new starts at $I$ is at most 
	$$
	8k! t n^{j-k-i}+2+n^{\beta/2} \le 8k! t n^{j-k-i}+n^{2\beta/3}.
	$$
	Taking a union bound over all possible $i$-sets $I$, with probability
	at least
	$$
	1-\binom{n}{i} \exp (-\Theta(n^{\beta/2})) = 1-\exp (-\Theta(n^{\beta/2})),
	$$
	every $i$-set is contained in at most
	\begin{equation}\label{eq:newstartsbound}
	8k!n^{j-k-i}t+n^{2\beta/3}
	\end{equation}
	new starts.
	
	\textbf{Jumps:} From each $j$-set $J$ which became active in the search process,
	but which did not contain $I$,
	if we queried a $k$-set containing $I$ and this $k$-set was an edge,
	then the degree of $I$ may increase by up to
	$\binom{k-i}{j-i}$.
	To bound the number of such jumps, we distinguish according to the intersection $Z=J\cap I$,
	and denote $z:=|Z|$.
	Observe that $0\le z \le i-1$, and for each of the $\binom{i}{z}$ many $z$-sets
	$Z \subset I$,
	by the fact that~\ref{stopdegree} has not been previously applied for this set $Z$,
	there are
	at most $d_t(Z)\le \frac{C_{k,j,z}t}{n^{k-j+z}}+n^\beta$
	many $j$-sets in $D_t$ which intersect $I$ in $Z$.
	For each such $j$-set $J$,
	there are at most $\binom{n}{k-j-i+z}\le n^{k-j-i+z}$ many $k$-sets containing both $J$ and $I$, i.e.\ which we might have
	queried from $J$ and which would result in jumps to $I$.
	
	Thus in total, the number of $k$-sets which we may have queried and which might have resulted in a jump to $I$
	is at most
	\begin{align*}
	\sum_{z=0}^{i-1} \binom{i}{z} \left(\frac{C_{k,j,z}t}{n^{k-j+z}}+n^\beta\right) n^{k-j-i+z}
	& = \sum_{z=0}^{i-1} \binom{i}{z} \left(\frac{C_{k,j,z}t}{n^{i}}+n^{k-j-i+z+\beta}\right) \\
	& \le 2^i\left(\max_{0\le z \le i-1}C_{k,j,z} \frac{t}{n^i} + n^{k-j-1+\beta}\right)\\
	& = 2^i\left(C_{k,j,i-1} \frac{t}{n^i} + n^{k-j-1+\beta}\right)=:N,
	\end{align*}
	since we chose $C_{k,j,j-1} \gg C_{k,j,j-2} \gg \ldots \gg C_{k,j,0}$.
	Then the number of edges that we discover which result in jumps to $I$ is dominated by
	$\Bi (N,p)$.
	By Lemma~\ref{lem:Chernoffwitherror}, with probability at least
	$1-\exp(-\Theta(n^{\beta/2}))$ this random variable is at most
	\begin{align*}
	2Np+n^{\beta/2}
	& \le \frac{(k-j)!}{\binom{k-j}{a}}2^{i+2}C_{k,j,i-1}\frac{t}{n^{k-j+i}}
		+ O\left(n^{-1+\beta}\right) + n^{\beta/2} \\
	& \le \frac{(k-j)!}{\binom{k-j}{a}}2^{i+2}C_{k,j,i-1}\frac{t}{n^{k-j+i}} + 2n^{\beta/2},
	\end{align*}
	and so the contribution to the degree of $I$
	made by jumps to $I$ is at most
	\begin{equation}\label{eq:jumpscontribution}
	(k-j)!2^{i+2}C_{k,j,i-1} \frac{t}{n^{k-j+i}}+n^{2\beta/3}.
	\end{equation}

	\textbf{Pivots:}
	Whenever we have a jump to $I$ or a new start at $I$, some $j$-sets containing $I$ become active.
	From these $j$-sets we may query further $k$-sets, potentially resulting in some more $j$-sets
	containing $I$ becoming
	active. However, the number of such $j$-sets containing $I$ that become active due to such a pivot
	is certainly at most $\binom{k-j}{a}$. Thus the number of further $j$-sets that become
	active due to pivots from some $j$-set $J$ is at most $\binom{k-j}{a}\cdot\Bi\left(\binom{n-j}{k-j},p\right)$,
	which has expectation $1+\eps$.
	
	Furthermore, the number of such sequential pivots that we may make
	before leaving $I$ in the $j$-tight path is $\lfloor \frac{k-i}{k-j}\rfloor \le k-i$.
	Thus the number of pivots arising from a single $j$-set containing $I$
	may be upper coupled with a branching process in which vertices
	in the first $(k-i)$ generations produce $\binom{k-j}{a}\cdot\Bi\left(\binom{n-j}{k-j},p\right)$ children,
	and thereafter no more children are produced.
	
	We bound the total size of all such branching processes together.
	Suppose the contribution to the degree of $I$ made by jumps and new starts
	is $x$. Then we have $x$ vertices in total in the first generation,
	and by the arguments above, with probability $1-\exp(-\Omega(n^{\beta/2}))$
	we have, by~\eqref{eq:newstartsbound} and~\eqref{eq:jumpscontribution}, that

	\[x \le \left(8k! + (k-j)!2^{i+2}C_{k,j,i-1}\right) \frac{t}{n^{k-j+i}}+2n^{2\beta/3}
	 \le 2^{i+3}k!C_{k,j,i-1} \frac{t}{n^{k-j+i}}+2n^{2\beta/3}.\]

	For convenience, we will assume (for an upper bound) that 
	in fact $x\ge n^\beta$.
	The number of children in the second generation is dominated
	by $\binom{k-j}{a}\cdot\Bi \left(x\binom{n-j}{k-j}, p\right)$, which has expectation
	$(1+\eps)x$, and so by Lemma~\ref{lem:Chernoffwitherror},
	with probability $1-\exp(-\Omega(n^{\beta/2}))$,
	the number of children is at most $2(1+\eps)x+n^{\beta/2} \le 4x$.
	Similarly, with probability $1-\exp(-\Omega(n^{\beta/2}))$,
	the number of vertices in the third generation is at most $16x$,
	and inductively the number of vertices in the $m$-th generation
	is at most $2^{2(m-1)}x$ for $1\le m \le k-i+1$.
	Thus in total, with probability at least $1-\exp(-\Theta(n^{\beta/2}))$,
	the number of vertices in total in all these branching processes
	is at most
	$$
	\sum_{m=1}^{k-i} 2^{2(m-1)}x \le 2^{2k}x \le 2^{3k+3}k!C_{k,j,i-1} \frac{t}{n^{k-j+i}}+n^{\beta}.
	$$
	
	However, the vertices in the branching process exactly
	represent (an upper coupling on) the $j$-sets which can be discovered
	due to jumps to or new starts at $I$ and
	the pivots arising from them, which are
	all of the $j$-sets containing $I$ which we discover in the \pathfinder\ algorithm.
	Thus with probability at least $1-\exp(-\Theta(n^{\beta/2}))$,
	the number of $j$-sets containing $I$ which became active is
	at most
	$$
	2^{3k+3}k!C_{k,j,i-1} \frac{t}{n^{k-j+i}}+n^{\beta} \le C_{k,j,i}\frac{t}{n^{k-j+i}}+n^{\beta},
	$$
	since we chose $C_{k,j,i}\gg C_{k,j,i-1}$.
	Taking a union bound over
	all $\binom{n}{i}$ many $i$-sets $I$, and observing that
	$\binom{n}{i}\exp(-\Theta(n^{\beta/2})) = o(1)$,
	the result follows.
\end{proof}

\begin{proof}[Proof of Lemma~\ref{lem:mainlemmastoppingtime}]
The statement of Lemma~\ref{lem:mainlemmastoppingtime}
follows directly from Proposition~\ref{prop:newstartsbound}
and Lemma~\ref{lem:SCdegrees}.
\end{proof}

\section{Longest paths: proof of Theorem~\ref{thm:mainresult}}\label{sec:mainproof}

The various statements contained in Theorem~\ref{thm:mainresult} have now all been proved,
with the exception of the upper bounds, whose standard proofs we delay
to the appendices.
\begin{itemize}
\item The upper bounds of statements~\ref{item:thm-subcritical},~\ref{item:thm-supercritical}
and~\ref{item:thm-loose} of Theorem~\ref{thm:mainresult} can be proved using a basic first moment method.
The details can be found in Appendix~\ref{app:firstmoment}.
\item The lower bound of statement~\ref{item:thm-subcritical}
follows directly from Lemmas~\ref{lem:lowerbound-subcritical:case1} and~\ref{lem:lowerbound-subcritical:case2}.
\item The lower bound of statement~\ref{item:thm-supercritical} is implied by Lemma~\ref{lem:mainlemmastoppingtime}, which is identical except that it omits the assumption that $\delta \gg \frac{\ln n}{\eps^2 n}$.
\item The lower bound of statement~\ref{item:thm-loose} is precisely Lemma~\ref{lem:loosesupercriticallower}.
\end{itemize}

\section{Concluding remarks}\label{sec:concluding}

Theorem~\ref{thm:mainresult} provides various bounds on the length $L$ of
the longest $j$-tight path, but these bounds may not be best possible.
Let us examine each of the three cases in turn.

\subsection{Subcritical case}

Here we proved the bounds
$$
\frac{j \ln n- \omega + 3 \ln \eps}{- \ln (1 - \eps)} \leq L \leq \frac{j \ln n+ \omega}{-\ln(1-\eps)}.
$$
A more careful version of the first moment calculation
implies that
if $\ell=\frac{j\ln n +c}{-\ln(1-\eps)}$ for some
constant $c\in \mathbb{R}$, then the expected number of paths of length $\ell$ is asymptotically
$d \cdot e^{c}$, where $d= \frac{(a!b!)^\ell}{z_\ell}= \frac{b!\left(a!b!\right)^s}{2((k-j)!)^{2s}}$.
This suggests heuristically that in this range, the probability that $X_\ell=0$, i.e.\ that there
are no paths of length $\ell$, is a constant bounded away from both $0$ and $1$,
and that in fact the bounds on $L$ are best possible up to the $3\ln \eps$ term in the lower bound.
This term is negligible (and can be incorporated into $\omega$) if $\eps$ is constant,
but as $\eps$ decreases, it becomes more significant. The term arises because as $\eps$ decreases,
the paths become longer, meaning that there are many more pairs of possible paths whose
existences in $H^k(n,p)$ are heavily dependent on one another, and the second moment method
breaks down. Thus to remove the $3\ln \eps$ term in the lower bound requires some new ideas.

\subsection{Supercritical case for $j\ge 2$}

In this case, we had the bounds
$$
(1 - \delta)\frac{\eps n}{(k-j)^2} \leq L \leq (1 + \delta)\frac{2 \eps n}{(k-j)^2}.
$$
Since in particular we may assume that $\delta\ll 1$, the upper bound (provided by the first moment method) and
the lower bound (provided by the analysis of the \pathfinder\ algorithm)
differ by approximately a factor of $2$.

One possible explanation for this discrepancy comes from the fact that we do
not query a $k$-set if it contains some explored $j$-set. As previously explained,
this condition is not necessary to guarantee the correct running of the algorithm,
but it is fundamentally necessary for our \emph{analysis} of the algorithm,
since it ensures that no $k$-set is queried twice and therefore each query is independent.

Removing this condition would allow us to try out many different paths with the
same end (i.e.\ different ways of reaching the same destination), which could
potentially lead to a longer final path since different sets of vertices
are used in the current path and are therefore forbidden for the continuation.

It is not hard to prove that the length $\ell$ of the current path in the modified
algorithm would very quickly reach almost $\frac{\eps n}{(k-j)^2}$
(i.e.\ our lower bound). For each possible way of reaching this, it is extremely
unlikely that the path can be extended significantly, and in particular
to length $\frac{2\eps n}{(k-j)^2}$. However, since there will be
very many of these paths, it is plausible that at least one of them
may go on to reach a larger size, and therefore our lower bound may not be best possible.

On the other hand, it could be that our upper bound is not best possible,
i.e.\ that the first moment heuristic does not give the correct threshold
path length. This could be because if there is one very long path, there are likely
to be many more (which can be obtained by minor modifications),
and so we may not have concentration around the expectation.

Therefore further study is required to determine the asymptotic value of $L$ more precisely.

\subsection{Supercritical case for $j=1$}

For loose paths, we proved the bounds
 \[
 (1 - \delta)\frac{\eps^2 n}{4(k-1)^2} \leq L \leq (1 + \delta)\frac{2 \eps n}{(k-1)^2},
 \]
which differ by a factor of $\Theta(\eps)$. In view of the supercritical case for $j\ge 2$, when
the longest path is of length $\Theta(\eps n)$ one might naively expect this to be the
case for $j=1$ as well, and that the lower bound is incorrect simply because
the proof method is too weak for $j=1$.

However, this is not the case for graphs, i.e.\ when $k=2$ and $j=1$, when the longest path
is indeed of length $\Theta(\eps^2 n)$.
The analogous result for general $k$ and $j=1$ was recently achieved by
Cooley, Kang and Zalla~\cite{CooleyKangZalla21}, who proved an upper bound of
approximately $\frac{2\eps^2 n}{(k-1)^2}$ by bounding the length of the longest
loose \emph{cycle} (via consideration of an appropriate $2$-core-like structure) and using
a sprinkling argument.
Nevertheless, this leaves a multiplicative factor of $8$ between the upper and lower bounds, which it would 
be interesting to close.

\subsection{Critical window}

One might also ask what happens when $\eps$ is smaller than allowed here,
i.e.\ when $\eps^3 n \nrightarrow \infty$.
In the case $j=1$, the lower bounds in the subcritical and supercritical case,
of orders $\frac{\ln(\eps^3 n)}{\eps}$ and $\eps^2 n$ respectively,
would both be
$\Theta(n^{1/3})$ when $\eps^3 n=\Theta(1)$, which suggests that this may indeed be the
correct critical window when $j=1$. However, for $j\ge 2$,
the bounds differ by approximately a factor of $n^{1/3}$ when $\eps^3 n = \Theta(1)$.
It would therefore be interesting to examine whether the statement of Theorem~\ref{thm:mainresult}
remains true for $j\ge 2$ even for smaller $\eps$.

\section*{Acknowledgement}
The collaboration leading to this paper was made possible partially by the support of
the Heilbronn Institute for Mathematical Research and EPSRC grant EP/P032125/1 during the workshop ``Structure and randomness in hypergraphs'', 17-21 December
2018, LSE London.

\bibliographystyle{plain}
\bibliography{References}

\newpage

\appendix

\section{Proofs of auxiliary results}\label{app:auxiliaryproofs}

In this appendix we will prove various auxiliary results that were stated
without proof in the paper. Note that the proofs of the auxiliary results
from Section~\ref{sec:secondmoment} appear separately
in Appendix~\ref{app:secondmomentcase1}, since they are thematically linked.

\subsection{Automorphisms}\label{app:partsizes}

\begin{proof}[Proof of Claim~\ref{claim:partsizes}]
We certainly have
$$
|F_i|=|e_i\setminus e_{i+1}| = |e_i|-|e_i\cap e_{i+1}| = k-j,
$$
and similarly $|G_i|=k-j$.
Furthermore,
$$
|A_i|=|e_i \cap e_{i+s}| = k-s(k-j) = k-\left(\left\lceil \frac{k}{k-j}\right\rceil-1\right) (k-j),
$$
so we have $1\le |A_i|\le k-j$ and $|A_i|\equiv k \mod k-j$,
which recall from~\eqref{eq:def:a} was precisely the definition of $a$, so $|A_i|=a$.
Finally, observe that $A_i \cup B_i = e_{i+s}\setminus e_{i+s-1}$,
and so
$$|B_i|= k-j - |A_i| = k-j-a \stackrel{\eqref{eq:def:b}}{=} b,$$
as required.
\end{proof}

\subsection{Chernoff bound}\label{app:chernoff}
\begin{proof}[Proof of Lemma~\ref{lem:Chernoffwitherror}]
We distinguish two cases. 
	
	\subsubsection*{Case 1: $Np>n^{\alpha}$} By applying~\eqref{eqn:chernoffstd} 
	with $\xi=Np$, we obtain
	\[
	\mathbb{P}(X\geq 2Np+n^\alpha) \le \mathbb{P}(X\geq 2Np) \stackrel{\eqref{eqn:chernoffstd}}{\leq}\exp\left(-\frac{(Np)^2}{\frac{8}{3}Np}\right)\leq\exp(-\Theta(n^{\alpha})),
	\]
	as required.
	
	\subsubsection*{Case 2: $Np\leq n^{\alpha}$}
	By applying~\eqref{eqn:chernoffstd} with $\xi=n^{\alpha}$, we obtain
	\begin{align*}
	\mathbb{P}(X\geq 2Np+n^{\alpha})&\leq \mathbb{P}(X\geq Np+n^{\alpha})\\
	& \stackrel{\eqref{eqn:chernoffstd}}{\leq} \exp\left(-\frac{n^{2\alpha}}{2(n^{\alpha}+n^{\alpha}/3)}\right)\\
	&=\exp(-\Theta(n^{\alpha})),
	\end{align*}
	which proves the assertion in this case. 
\end{proof}

\subsection{First moment method} \label{app:firstmoment}

In this section we prove the upper bounds in
all three statements of Theorem~\ref{thm:mainresult}.
For convenience, we restate these bounds in the following lemma.

\begin{lemma}\label{lem:upperbound}
	Let $k,j\in\NN$ satisfy $1 \leq j \leq k-1$. Let $a\in\NN$ be the unique integer satisfying
	$1 \leq a \leq k-j$ and $a \equiv k \bmod (k-j)$.
	Let $\eps = \eps(n) \ll 1$ satisfy $\eps^3 n \xrightarrow{n\to\infty} \infty$
	and let
	$$
	p_0 = p_0(n;k,j) := \frac{1}{\binom{k-j}{a} \binom{n-j}{k-j}}.
	$$
	Let $L$ be the length of the longest $j$-tight path in $H^k(n,p)$.

	\begin{enumerate}[label=\normalfont{(\roman*)}]
		\item \label{item:lem-upper-subcritical}
		If $p = \frac{1-\eps}{\binom{k-j}{a} \binom{n-j}{k-j}}$, then \whp\
		\[
		L \leq \frac{j \ln n+ \omega}{-\ln(1-\eps)} ,
		\]
		for any $\omega = \omega(n) \xrightarrow{n\to\infty} \infty$.
		\item \label{item:lem-upper-supercritical}
		If $p = \frac{1+\eps}{\binom{k-j}{a} \binom{n-j}{k-j}}$, then
		for any $\delta$ satisfying $\delta \gg \max\{\eps,\frac{\ln n}{\eps^2 n}\}$,
		\whp
		\[
		L \leq (1 + \delta)\frac{2 \eps n}{(k-j)^2}.
		\]
	\end{enumerate}
\end{lemma}

Note that
the only difference between this statement and the upper bounds in Theorem~\ref{thm:mainresult}
is that in Theorem~\ref{thm:mainresult}~\ref{item:thm-loose}
we assume $\delta^2 \eps^3 n \to \infty$ in place of $\delta \gg \frac{\ln n}{\eps^2 n}$,
but it is easy to see that the former condition implies the latter.

\begin{proof}

Since
\[\prob(L\geq\ell) = \prob(\hat{X}_{\ell}\geq1) \leq \expec(\hat{X}_{\ell})\]
by Markov's inequality, it suffices to show that $\expec(\hat{X}_{\ell}) \xrightarrow{n\to \infty} 0$
for the relevant values of $\ell$ and $p$.

We first prove the subcritical case (i.e.~\ref{item:lem-upper-subcritical}),
so we set $p = \frac{1-\eps}{\binom{k-j}{a} \binom{n-j}{k-j}}$
and $\ell = \frac{j \ln n+ \omega}{-\ln(1-\eps)}$.
It is convenient to assume that $\omega=o(\ln n)$, which is permissible
since the statement becomes stronger for smaller~$\omega$. With this assumption we have
$\ell = \Theta \left(\frac{\ln n}{\eps}\right)= o(n)$.
Then by Corollary~\ref{cor:expectationupperbound},
\begin{align*}
\expec(\hat{X}_{\ell})
& = \Theta(1)\frac{(n)_{v}\left(1-\eps\right)^{\ell}}{\left(a!b!\binom{k-j}{a}\binom{n-j}{k-j}\right)^{\ell}}
 \le \Theta(1)\frac{n^v(1-\eps)^{\ell}}{(n-k)^{\ell(k-j)}}\\
& \le \Theta(1) \left(1+\frac{k}{n-k}\right)^{\ell(k-j)} n^j(1-\eps)^{\ell}\\
& = \Theta(1) \left(1+ O\left(\frac{\ell}{n}\right)\right) \exp(j\ln n + \ell \ln(1-\eps))\\
& = \Theta(1)(1+o(1)) \exp(-\omega) \to 0,
\end{align*}
which completes the proof of~\ref{item:lem-upper-subcritical}.

It remains to prove~\ref{item:lem-upper-supercritical}, for which
we set $p = \frac{1+\eps}{\binom{k-j}{a} \binom{n-j}{k-j}}$
and $\ell = (1 + \delta)\frac{2 \eps n}{(k-j)^2}$.
Observe that $v=(k-j)\ell+j= \Theta(\eps n)\leq\frac{n}{2}$.
By applying Stirling's formula
we obtain
\begin{align*}
(n)_{v}&=\frac{n!}{(n-v)!}\\
&=(1+o(1))\sqrt{\frac{n}{n-v}}\frac{n^v}{e^v}\left(1+\frac{v}{n-v}\right)^{n-v} \\\
&=O\left(\frac{n^v}{e^v}\exp\left((n-v)\left(\frac{v}{n-v}-\frac{v^2}{2(n-v)^2} + O\left(\frac{v^3}{(n-v)^3}\right)\right)\right)\right)\\
&=O\left(n^v\exp\left(\frac{-v^2}{2(n-v)}+O\left(\frac{v^3}{n^2}\right)\right)\right) \\
&=O\left(n^v\exp\left(-\frac{\ell^2(k-j)^2+O(\ell)}{2n(1+O(\eps))}+O(\eps^3n)\right)\right)\\
&=O\left(n^v\exp\left(-\frac{\ell^2(k-j)^2}{2n}+O(\eps^3n)\right)\right),
\end{align*}
where in the last line we have used the fact that $\ell/n = O(\eps) = O(\eps^3 n)$.
Therefore by Corollary~\ref{cor:expectationupperbound}, we have
\begin{align*}
\mathbb{E}(\hat{X_{\ell}}) &= \frac{O\left(n^v\exp\left(-\frac{\ell^2(k-j)^2}{2n}+O(\eps^3n)\right)\right)}{(a!b!)^{\ell}}\left(\frac{1+\eps}{\binom{n-j}{k-j}\binom{k-j}{a}}\right)^{\ell} \\
&=  O\left(n^j\exp\left(O(\eps^3 n)\right)  \left(\frac{n^{k-j} \exp\left(\frac{-\ell(k-j)^2}{2n}\right)(1+\eps)}{\left(1+O\left(\frac{1}{n}\right)\right)n^{k-j}}\right)^{\ell} \; \right)  \\
&= O\left(n^j\exp\left(O(\eps^3 n)\right)\left(1+O\left(\frac{\ell}{n}\right)\right)  \bigg(\exp\Big(-(1+\delta)\eps\Big)(1+\eps)\bigg)^{\ell} \; \right).
\end{align*}
Now recall that $1+O(\ell/n) = 1+O(\eps)=O(1)$,
and furthermore
\begin{align*}
\exp(-(1+\delta)\eps)(1+\eps) & = \exp\left(-(1+\delta)\eps + \eps +O\left(\eps^2\right)\right)\\
& = \exp\left(-\delta\eps + O\left(\eps^2\right)\right)
 \le \exp\left(\frac{-\delta\eps}{2}\right),
\end{align*}
since $\delta \gg \eps$.
Therefore
\begin{align*}
\mathbb{E}(\hat{X_{\ell}})
&=O\left(n^j\exp\left(O\left(\eps^3 n\right)-\ell\delta\eps/2\right)\right)\\
&= O\left(\exp\left(-\Theta\left(\eps^2\delta n\right) +j\ln n\right)\right)\to 0,
\end{align*}
by the fact that $\delta \gg \frac{\ln n}{\eps^2 n}$.
This completes case~\ref{item:lem-upper-supercritical}. 
\end{proof}

\subsection{Algorithm properties}\label{app:algprops}

\begin{proof}[Proof of Proposition~\ref{proposition:usingchernoff}]
	Using~\eqref{eqn:discoveredviaqueries}, the stated inequality is equivalent to
	\[ (1-\gamma) pt \le \sum_{i=1}^t X_i \le (1 + \gamma) p t. \]
	By the Chernoff bounds of Lemma~\ref{lem:chernoffbounds},
	the probability that one of these inequalities fails is at most
	\begin{align*}
	\exp \left( - \frac{(\gamma pt)^2}{2pt} \right) + \exp \left( - \frac{(\gamma pt)^2}{2pt +\gamma pt} \right)
	= \exp\left(-\Theta(\gamma^2 pt)\right),
	\end{align*}
as required.
\end{proof}

\begin{proof}[Proof of Proposition~\ref{prop:loosechernoff}]
	Since $|A_t \cup E_t| \ge |S_t|$ and $(k-1)pt = (1+\eps)t/\binom{n-1}{k-1}$, we can apply Proposition~\ref{proposition:usingchernoff}:
	it is enough to find $\gamma$ such that $\gamma = o(\delta \eps)$ and $\gamma^2 p t \rightarrow \infty$.
	Recall that $\delta^2\eps^3 n\to \infty$.
	Let $\omega = \delta^2 \eps^3 n$.
	Then $\gamma = \delta \eps / \omega^{1/3}$ clearly satisfies $\gamma = o(\delta \eps)$.
	On the other hand, by the choice of $t = T_0$, we have $pt = \Theta(\eps n)$.
	Thus $\gamma^2 p t = \Theta(\omega^{1/3}) \rightarrow \infty$, as required.
\end{proof}

\section{Second moment method: Case~1}\label{app:secondmomentcase1}

In this appendix we prove the auxiliary results required
for the proof of Lemma~\ref{lem:lowerbound-subcritical:case1},
i.e.\ the second moment method for the case when $j\le k-2$ or $j=k-1=1$.

\begin{proof}[Proof of Claim~\ref{claim:expectationfirsttermsplit}]

Observe that 
\begin{align*}
\expec(X_{\ell}^2) &= \sum_{(A,B)\in\cP_{\ell}^2} \prob(A,B\subset H^k(n,p)) \\
&= \sum_{q,r,\mathbf{c}} |\cP_{\ell}^2(q,r,\mathbf{c})|p^{2\ell-q}. 
\end{align*}
Furthermore, observe that in the case $q=0$, we must have $r=0$ and $\mathbf{c}=()$
an empty sequence. In this case, we have
$$
|\cP_\ell^2(0,0,())|\le (n)_v^2,
$$
while for $q\ge 1$ clearly $\cP_{\ell}^2(q,r,\mathbf{c})$ is empty unless also $r\ge 1$, and the result follows.
\end{proof}

\begin{proof}[Proof of Proposition~\ref{prop:pl2qrcbound}]

To estimate $|\cP_\ell^2(q,r,\mathbf{c})|$, we will
regard $A$ and $B$ as $j$-tight paths of length $\ell$ which must be
\emph{embedded} into $K_n^{(k)}$ subject to certain restrictions
(so that the parameters $q,r,\mathbf{c}$ are correct),
and estimate the number of ways of performing this embedding appropriately.
We will denote the edges of $A$ by
$(e_1,\ldots,e_{\ell})$ and the edges of $B$ by $(f_1,\ldots,f_{\ell})$, each in the natural order.

First we embed the path $A$; there are $(n)_{v}$ ways of choosing its vertices in order.
Then we embed the path $B$ subject to certain restrictions, since we must obtain
the parameters $q,r,\mathbf{c}$.
We first choose which of the edges $f_i$ on $B$ will lie in $Q(A,B)$---recall
that the $i$-th interval must be of length $c_i$, and therefore must have the form
$(f_{t_i},\ldots,f_{t_i+c_i-1})$, for some $1\le t_i \le \ell-c_i+1$.
Thus the $i$-th interval is determined by the choice of its
first edge $f_{t_i}$. Having already chosen intervals of length
$c_1,\ldots,c_{i-1}$, there are only $\ell-c_1-c_2-\dotsb -c_{i-1}$
edges of $B$ left, of which certainly the last $c_{i}-1$ cannot be
chosen for $f_{t_i}$, since then either the interval would intersect
with another previously chosen interval, or it would extend beyond the
end of $B$. Thus there are at most
$\ell-c_1-\dotsb-c_i +1$ possible choices for $f_{t_i}$.
Subsequently, we choose which edges of $A$ to embed this interval onto.
The corresponding interval in $A$ must have the form either
$$
(e_{s_i},\ldots,e_{s_i+c_i-1})
$$
or
$$
(e_{s_i},\ldots,e_{s_i-c_i+1}),
$$
depending on whether the orientation is with
or against the direction on $A$. There are $2$ choices for the orientation,
and subsequently (arguing as for $B$) at most $\ell-c_1-\dotsb-c_i+1$ choices for $e_{s_i}$.

Thus the number of ways of choosing where to embed the edges
of $Q(A,B)$ is at most
\begin{equation}\label{eqn:embededgesinA}
    \prod_{i=1}^r (\ell-c_1-\dotsb -c_i +1)2(\ell-c_1- \dotsb -c_i+1) \le 2^r(\ell-q+1)^2\ell^{2(r-1)},
\end{equation}
where we have used the fact that $c_1+ \dotsb + c_r=q$.
Observe here that we may well have overcounted: for an interval of length
$1$, the factor of $2$ is superfluous, since orientation makes no difference;
furthermore, if $r>1$, then having embedded the first interval is often
more restrictive with respect to where the second may be embedded than
we accounted for. However, this expression is certainly
an upper bound.

Note also that for $i\le r-1$ we use the crude bound $\ell-c_1-\dotsb -c_i+1\le \ell$,
whereas we are more careful about $c_r$. The reason is that in the case
$r=1$ we will have to bound terms rather precisely, whereas for
$r\ge 2$ we will have plenty of room to spare in the calculations.

We have now fixed how the edges of intervals in $B$ are embedded onto intervals in $A$,
but we also need to account for different ways of ordering the vertices in these intervals.
Since the $i$-th interval forms a $j$-tight path of length $c_i$,
there are $z_{c_i}$ possible ways of re-ordering the vertices of $B$ along it,
but still embedding into $A$ in a way consistent with the edge-assignment.
This is true regardless of where the interval lies on $A$ or $B$,
even if it includes some of the head or tail of $A$ or $B$.

One difficulty is that two different intervals may share
vertices, and therefore not every re-ordering is admissible.
However, we may certainly use $z_{c_i}$ as an upper bound
for each $i$. Thus by~\eqref{eq:isomorphisms},
the number of ways of choosing where to embed the
vertices of $B$ within edges of $Q(A,B)$ is at most
\begin{equation}\label{eqn:vertexorderingsinA}
    \prod_{i=1}^r z_{c_i}=\prod_{i=1}^r\Theta\left((a!b!)^{c_i}\right)=(a!b!)^q\Theta(1)^r.
\end{equation}

We now need to bound the number of ways of embedding the remaining vertices of $B$,
for which we need a \emph{lower} bound on the number of vertices in edges of $Q(A,B)$,
i.e.\ vertices of $B$ which have already been embedded into $A$.
Let us first consider a simple \emph{upper} bound: the $i$-th interval
contains $(k-j)c_i+j$ vertices, and so
we have already embedded at most 
\begin{equation}\label{eqn:simpleupperboundvertexchoicesB}
\sum_{i=1}^r((k-j)c_i+j)=(k-j)q+rj
\end{equation}
vertices, with equality if and only if no two intervals share a vertex.
We find a \emph{lower} bound by considering when the intervals share as many vertices as possible.

Let us first consider the intervals in their natural order along $B$.
Then the number of vertices lying in two consecutive intervals
is at most the size of the intersection of two \emph{non-consecutive}
edges $|e_i\cap e_{i+2}| = \max\{k-2(k-j),0\}$.
Therefore the total number of vertices lying in more than one interval is at most
$$
(r-1)\cdot\text{max}\{k-2(k-j),0\} = (r-1)(j- \min\{j,k-j\}).
$$
Thus using~\eqref{eqn:simpleupperboundvertexchoicesB}, 
the number of vertices already embedded is at least
\[
T(r)= 
(k-j)q+j+(r-1)\min\{j,k-j\}.
\]
Therefore we have at most $v-T(r)$ vertices of $B$ still left to embed, for which there are at most
\begin{equation}\label{eq:embedrest}
(n)_{v-T(r)} \leq \frac{(n)_{v}}{(n-v)^{T(r)}}
\end{equation}
choices.

Combining~\eqref{eqn:embededgesinA},~\eqref{eqn:vertexorderingsinA} and~\eqref{eq:embedrest}
with the fact that there were $(n)_{v}$ ways of embedding $A$, for $r\ge 1$
we obtain
\begin{align*}
|\cP_\ell^2(q,r,\mathbf{c})| & \le (n)_{v}2^r(\ell-q+1)^2 \ell^{2(r-1)} (a!b!)^q \Theta(1)^r \frac{(n)_{v}}{(n-v)^{T(r)}} \\
& = (n)_{v}^2 \frac{(\ell-q+1)^2 \ell^{2(r-1)}(a!b!)^q \Theta(1)^r}{(n-v)^{T(r)}}
\end{align*}
as claimed.
\end{proof}

\begin{proof}[Proof of Proposition~\ref{prop:doublesum}]
It will turn out that for each $q$, the $r=1$ term is the most significant, so we will treat this case separately. 
We define the following functions for each positive integer $q$ and $r\in [q]$.
\[ y_q(r) := \sum_{\substack{c_1+\dotsb+c_r=q \\ c_1\geq\dotsb\geq c_r\geq 1}}1 \]and
\[x_q(r):=
\frac{\ell^{2(r-1)}C^r}{(n-v)^{T(r)}}y_q(r)
\]
Combining these with~\eqref{eqn:squareexp}, we obtain
\begin{align}
\sum_{r=1}^{q}\ \sum_{\substack{c_1+\dotsb+c_r=q \\ c_1\geq\dotsb\geq c_r\geq 1}}\frac{(\ell-q+1)^2\ell^{2(r-1)} (a!b!)^qC^r}{p^q (n-v)^{T(r)}}
= 
\frac{(a!b!)^q}{p^q}(\ell-q+1)^2\sum_{r=1}^{q}x_q(r).
 \label{eqn:maintermforexpectation}
\end{align}

Observe that
\begin{equation}\label{eq:yratio}
y_q(r+1)\leq\sum_{c_{r+1}=1}^{q}y_{q-c_{r+1}}(r)\leq\sum_{c_{r+1}=1}^{q}y_{q}(r)\leq q\cdot y_q(r),
\end{equation}
and
\begin{equation}\label{eq:Tdiff}
T(r+1)-T(r)=\min\{j,k-j\},
\end{equation}
so for $r \in [q]$ we have
\begin{align*}
\frac{x_q(r+1)}{x_q(r)} & = \frac{\ell^2C}{(n-v)^{T(r+1)-T(r)}}\frac{y_q(r+1)}{y_q(r)}\\
&\leq \frac{\ell^2C}{(n-v)^{\min\{j,k-j\}}}q  = O\left(\frac{(\ln (\eps^3n^j))^3}{\eps^3 n^{\min\{j,k-j\}}}\right),
\end{align*}
where we have used the fact that $q\le \ell = O\left(\frac{\ln(\eps^3 n^j)}{\eps}\right)$.
Now let us observe that in the case $j=1$, setting $\lambda:= \eps^3 n$ which tends to infinity by assumption,
we have
$$
\frac{x_q(r+1)}{x_q(r)} = O\left(\frac{(\ln \lambda)^3}{\lambda}\right) =O\left(\lambda^{-1/2}\right).
$$
On the other hand, if $j\ge 2$, then (since we are in Case 1) we also have $j\le k-2$, and so
$$
\frac{x_q(r+1)}{x_q(r)} = O\left(\frac{(\ln n)^3}{\eps^3 n^2}\right) =O\left(n^{-1/2}\right).
$$

Setting
$$
w:=
\begin{cases}
\lambda^{1/2} & \mbox{if } j=1, \\
n^{1/2} & \mbox{if } 2\le j\le k-2,
\end{cases}
$$
we have $w \to \infty$ and $\frac{x_q(r+1)}{x_q(r)} = O(1/w)$ in all cases.%
\footnote{\label{footnote:prooffail}Note that it is here that the argument fails for $2\le j=k-1$, since
we would only obtain the bound
$$
\frac{x_q(r+1)}{x_q(r)} = O\left(\frac{(\ln n)^3}{\eps^3 n}\right),
$$
and if $\eps$ is very small (i.e.\ $\eps^3 n \to \infty$ very slowly), this may not tend to zero.
If we were to assume the slightly stronger condition of $\frac{\eps^3 n}{(\ln n)^3}\to \infty$
in Theorem~\ref{thm:mainresult},
then this would not be an issue and we would not need to handle the case
$2\le j=k-1$ separately.
}
Therefore, we obtain
\begin{align*}
\sum_{r=1}^{q}x_q(r) 
&=x_q(1) \left(1 + \sum_{i=1}^{q-1}O\left(\frac{1}{w}\right)^{i} \right) \nonumber \\
&= \frac{C y_q(1)}{(n-v)^{T(1)}}\cdot (1+o(1))   \\
& \le \frac{2C}{n^{(k-j)q+j}}.
\end{align*}
Substituting this upper bound into~\eqref{eqn:maintermforexpectation} gives
\begin{align*}
\sum_{r=1}^{q}\ \sum_{\substack{c_1+\dotsb+c_r=q \\ c_1\geq\dotsb\geq c_r\geq 1}}\frac{(\ell-q+1)^2\ell^{2(r-1)} (a!b!)^qC^r}{p^q (n-v)^{T(r)}}
& \le 
\frac{(a!b!)^q}{p^q}(\ell-q+1)^2 \frac{2C}{n^{(k-j)q+j}} \\
& = O\left(n^{-j}\right) (\ell-q+1)^2 \left(\frac{a!b!}{p n^{k-j}}\right)^q \\
& = O\left(n^{-j}\right) \frac{(\ell-q+1)^2}{(1-\eps)^q},
\end{align*}
as required.    
\end{proof}

\begin{proof}[Proof of Claim~\ref{claim:roneterm}]

By a change of index $i=\ell-q$, we get
\begin{align*}
\sum_{q=1}^{\ell}\frac{(\ell-q+1)^2}{(1-\eps)^q} 
&= \sum_{i=0}^{\ell-1}\frac{(i+1)^2}{(1-\eps)^{\ell-i}}\\
&\leq(1-\eps)^{-\ell} \sum_{i=-2}^{\infty}(i+1)(i+2)(1-\eps)^{i}\\
&\leq(1-\eps)^{-\ell} \frac{\mathrm{d}^2}{\mathrm{d}\eps^2}\left(\sum_{i=-2}^{\infty}(1-\eps)^{i+2}\right)\\
& = (1-\eps)^{-\ell} \frac{\mathrm{d}^2}{\mathrm{d}\eps^2} \left(\frac{1}{\eps}\right) \\
& = \frac{2(1-\eps)^{-\ell}}{\eps^3}
\end{align*}
as claimed.
\end{proof}

\begin{proof}[Proof of Claim~\ref{claim:secondmoment}]
Using Proposition~\ref{prop:doublesum} and Claim~\ref{claim:roneterm},
together with the fact that $\ell=\frac{j\ln n-\omega + 3 \ln \eps}{-\ln(1-\eps)}$,
we have 
\begin{align*}
\sum_{q=1}^\ell\ \sum_{r=1}^{q}\ \sum_{\substack{c_1+\dotsb+c_r=q \\ c_1\geq\dotsb\geq c_r\geq 1}}
& \frac{(\ell-q+1)^2\ell^{2(r-1)} (a!b!)^qC^r}{p^q (n-v)^{T(r)}}\\
& \stackrel{\eqref{eq:doublesum}}{=} O(n^{-j})\sum_{q=1}^{\ell}\frac{(\ell-q+1)^2}{(1-\eps)^q} \\
& \stackrel{\eqref{eqn:roneterm}}{=} O(n^{-j})\frac{(1-\eps)^{-\ell}}{\eps^3} \\
& = O(1)\exp(-j\ln n - 3\ln \eps - \ell \ln(1-\eps))\\
&= O(1)\exp(-\omega) = o(1).
\end{align*}
Substituting this into~\eqref{eqn:squareexp}, we obtain
$
\EE(X_\ell^2)=(n)_v^2 p^{2\ell}(1+o(1)),
$
as claimed.
\end{proof}

\section{Second moment method: Case~2}\label{app:secondmomentcase2}

In this appendix we prove Lemma~\ref{lem:lowerbound-subcritical:case2},
i.e.\ the second moment method for the case when $2\le j = k-1$.

Since much of the proof is identical to the proof of Lemma~\ref{lem:lowerbound-subcritical:case1},
rather than repeating the argument, we will show how to adapt the previous proof
to the special case when $2\le j = k-1$.
Recall from Footnote~\ref{footnote:prooffail} that the reason the proof did not go through for this case was that in~\eqref{eq:Tdiff} we have
$T(r+1)-T(r) = \min \{j,k-j\}=1$,
and we obtain a single factor of $n$ in the denominator of $\frac{x_{q}(r+1)}{x_q(r)}$,
which is not quite enough to dominate the $\ell^2 q \le \ell^3$ term in the numerator.

However, recall that $T(r)=T_q(r)$ represents a lower bound on the number of vertices of $B$ already
embedded in $Q(A,B)$ if this set splits into $r$ intervals (for given $q$).
To help illustrate the main idea in the adaptation of the previous proof,
let us compare $T_q(2)$ with $T_q(1)$.
We have $T_q(1)=(k-j)q + k-1=T_q(2)-1$, but the only 
way of having two intervals which partition $q$ edges and which together contain
exactly $(k-j)q+k$ vertices
is for the two intervals to have exactly one edge separating them,
i.e.\ for the intervals to be of the form
$f_{t_1},\ldots,f_{t_2}$ and $f_{t_2+2},\dotsb,f_{t_3}$.%
\footnote{Observe that it is indeed possible to have two such intervals without
the edge $f_{t_2+1}$ between them also being shared, since the order of vertices either side
of the separating edge may be different on $A$ and $B$.}
We call such a pair of intervals \emph{adjacent}.
Heuristically, if this is to happen then we have only one choice
for where to place the second interval, rather than the factor of $\ell$
that we obtained previously (in the arguments leading to~\eqref{eqn:embededgesinA}).
On the other hand, we must choose \emph{which} of the intervals
will be adjacent.

We therefore introduce a new parameter $r_1=r_1(A,B)$, which is the
number of pairs of intervals which are adjacent on $B$ (and therefore also on $A$),
and let $\cP_\ell^2(q,r,r_1,\mathbf{c})$ be the subset of $\cP_\ell^2$
with the appropriate parameters.
For convenience, define $r_2:=r-r_1$.

Instead of $T(r)=T_q(r)$ as in the previous case, we now define

\begin{align*}
T(r_1,r_2)=T_q(r_1,r_2)& =q+r(k-1)-r_1(k-2)-(r_2-1)(k-3)\\
& = q+r_1+2r_2+k-3.
\end{align*}

For convenience, we also define $r_1':= \max\{r_1,1\}$ (so $r_1'=r_1$ unless $r_1=0$). The analogue of Proposition~\ref{prop:pl2qrcbound} is the following.

\begin{proposition}\label{prop:pl2qrr1cbound}
For $q,r_1\ge 1$, there exists a constant $C$ such that
\begin{align*}
|\cP_\ell^2(q,r,r_1,\mathbf{c})| & \le (n)_{v}^2 \frac{(\ell-q+1)^2 \ell^{2(r_2-1)} C^r}{(n-v)^{T(r_1,r_2)}} \left( \frac{r^2}{r_1'}\right)^{r_1}.
\end{align*}
\end{proposition}

\begin{proof}

In contrast to Case~1, when choosing where to place the intervals of $Q(A,B)$ on $B$,
we first choose which pairs of intervals will be adjacent, and
in which order such a pair appears along $B$.
This is equivalent to choosing an auxiliary
\emph{adjacency graph} $G$, an oriented
graph whose vertices are the intervals of $Q(A,B)$, and where an edge oriented from $I_1$
to $I_2$ in $G$ indicates that these intervals will be adjacent and
that $I_1$ will be the first of these to appear in the natural
order along $B$.
The number of ways of choosing $r_1$ such directed edges from
among the $r$ intervals is at most
\begin{equation}\label{eq:chooseadjacencies}
\binom{\binom{r}{2}}{r_1} 2^{r_1} \le \left(\frac{e(r^2/2)}{r_1'}\right)^{r_1} 2^{r_1}
\le \left(\frac{er^2}{r_1'}\right)^{r_1}.
\end{equation}
Note that not every such choice is possible
because in fact $G$ must have maximum indegree and
maximum outdegree at most $1$, and furthermore must be acyclic.
However, this expression certainly gives an upper bound.

We now observe that the components of $G$ are simply directed paths
(including isolated vertices, which are paths of length $0$).
Furthermore, for every directed path in the adjacency graph,
choosing where on $B$ to place the first edge of the first interval
fixes the positions of all remaining edges of every interval on the path.
We therefore consider the intervals corresponding to a component of $G$
to be one \emph{super-interval} (including the isolated vertices of $G$,
which correspond to a single interval).
The length of a super-interval consisting of $I_{i_1},\ldots,I_{i_t}$
is
$$
c_{i_1}+\dotsb + c_{i_t} + t-1\ge c_{i_1}+\dotsb +c_{i_t},
$$
since the edge between
two adjacent intervals also belongs to the super-interval.
The number of super-intervals is $r-r_1=r_2$.

Now we choose where to place the super-intervals on $B$, and as before
we have at most $\ell$ choices for each, but for the last of the super-intervals
we use the stronger bound $\ell-q+1$, similarly to Case~1.
Thus the number of ways of choosing the super-intervals on $B$
is at most
\begin{equation}\label{eq:superintervalsonB}
\ell^{r_2-1}(\ell-q+1).
\end{equation}

We then need to choose where to place the super-intervals on $A$.
(Note that while the edge between two adjacent intervals is not
the same in $A$ and $B$, which edge of $A$ this is will naturally
be fixed by the choice of where the adjacent intervals, which must lie either side of it, have been placed on $A$.)
As before, for each super-interval we first choose an orientation
along $A$, and subsequently there are at most $\ell$ choices for where
to place the super-interval, or $\ell-q+1$ for the last super-interval.
Thus the number of ways of choosing where the super-intervals lie in $A$
is at most
\begin{equation}\label{eq:superintervalsonA}
2^{r_2} \ell^{r_2-1}(\ell-q+1).
\end{equation}

Furthermore, by~\eqref{eq:isomorphisms} the number of ways of
ordering the vertices within the $i$-th interval in a way that is
consistent with the choice of edges is at most
$$
z_{c_i} = \Theta\left((a!b!)^{c_i}\right) = \Theta(1),
$$
since $a=1$ and $b=0$. Since there are $r$ intervals in total,
the number of ways of re-ordering the vertices within $Q(A,B)$
is at most
\begin{equation}\label{eq:reorderings:case2}
\left(\frac{C}{2e}\right)^r
\end{equation}
for some sufficiently large constant $C$.
Thus combining the terms from~\eqref{eq:chooseadjacencies},~\eqref{eq:superintervalsonB}, \eqref{eq:superintervalsonA}
and~\eqref{eq:reorderings:case2}, the number of ways of choosing where on $A$ to embed the vertices
within $Q(A,B)$ is at most
$$
\left(\frac{er^2}{r_1'}\right)^{r_1} \ell^{2(r_2-1)} (\ell-q+1)^2 2^{r_2} \left(\frac{C}{2e}\right)^r
\le \left(\frac{r^2}{r_1'}\right)^{r_1} \ell^{2(r_2-1)} (\ell-q+1)^2 C^r.
$$
This replaces the terms $(\ell-q+1)^2 \ell^{2(r-1)} C^r$ from Proposition~\ref{prop:pl2qrcbound}.
All other terms remain the same as in Case 1, and observing that when $j=k-1$
we have $a=1$ and $b=0$, we obtain the statement of Proposition~\ref{prop:pl2qrr1cbound}.
\end{proof}

Now  the analogue of Corollary~\ref{cor:squareexp} is the following
\begin{corollary}\label{cor:squareexpcase2}
\begin{equation}\label{eqn:squareexpcase2}
\EE(X_\ell^2) 
\leq\mathbb (n)_{v}^2p^{2\ell}\left(1+
\sum_{q=1}^{\ell}\sum_{r=1}^{q}\sum_{r_1=0}^{r-1}\sum_{\substack{c_1+\dotsb+c_r=q \\ c_1\geq\dotsb\geq c_r\geq 1}}
\frac{(\ell-q+1)^2\ell^{2(r_2-1)} C^r}{p^q (n-v)^{T(r_1,r_2)}} \left( \frac{r}{r_1'}\right)^{r_1}  \right).
\end{equation}
\end{corollary}

The following takes the place of Proposition~\ref{prop:doublesum}.

\begin{proposition}\label{prop:triplesum}
$$
\sum_{r=1}^{q}\sum_{r_1=0}^{r-1}\sum_{\substack{c_1+\dotsb+c_r=q \\ c_1\geq\dotsb\geq c_r\geq 1}}
\frac{(\ell-q+1)^2\ell^{2(r_2-1)}C^r}{p^q (n-v)^{T(r_1,r_2)}}  \left(\frac{r^2}{r_1'}\right)^{r_1}
=
O\left(n^{-j}\right) \frac{(\ell-q+1)^2}{(1-\eps)^q}.
$$
\end{proposition}

\begin{proof}
We first observe that
\begin{align*}
T(r_1,r_2) & = q+r_1+2r_2+k-3 \\
& = q + 2r - r_1 + k-3 = O(\ell)
\end{align*}
and therefore
\begin{align*}
(n-v)^{T(r_1,r_2)} & = n^{q+2r-r_1 + k-3} \left(1-O\left(\frac{\ell}{n}\right)\right)^{O(\ell)}\nonumber\\
& = n^{q+2r-r_1 + k-3} \left(1-O\left(\frac{\ell^2}{n}\right)\right)\nonumber\\
& =  n^{q+2r-r_1 + k-3}(1-o(1)).
\end{align*}
Since $p=\frac{1-\eps}{n-k+1}$, we obtain
\begin{equation}\label{eq:ntermscase2}
p^q (n-v)^{T(r_1,r_2)} = (1+o(1))(1-\eps)^q n^{2r-r_1+k-3}.
\end{equation}
As in Case 1, we define
\[
y_q(r) := \sum_{\substack{c_1+\dotsb+c_r=q \\ c_1\geq\dotsb\geq c_r\geq 1}}1,
\]
but this time we define
$$
x_q(r):=y_q(r) \sum_{r_1=0}^{r-1} \frac{\ell^{2r-2r_1}C^r}{n^{2r-r_1}}\left(\frac{r^2}{r_1'}\right)^{r_1}
= y_q(r)\left(\frac{C\ell^2}{n^2}\right)^r \sum_{r_1=0}^{r-1} \left(\frac{nr^2}{\ell^2 r_1' }\right)^{r_1},
$$
so that substituting these definitions into the triple-sum and using~\eqref{eq:ntermscase2}, we obtain
\begin{align}\label{eqn:squareexpcase2simplified1}
& \sum_{r=1}^{q}\sum_{r_1=0}^{r-1}\sum_{\substack{c_1+\dotsb+c_r=q \\ c_1\geq\dotsb\geq c_r\geq 1}}
\frac{(\ell-q+1)^2\ell^{2(r_2-1)} C^r}{p^q (n-v)^{T(r_1,r_2)}} \left(\frac{r^2}{r_1'}\right)^{r_1} \nonumber\\
& \hspace{3cm} =
(1+o(1))\frac{(\ell-q+1)^2}{(1-\eps)^q \ell^2 n^{k-3}}\sum_{r=1}^{q}x_q(r).
\end{align}
Once again, the initial aim is to show that $\sum_{r=1}^{q}x_q(r) = (1+o(1))x_q(1)$. To achieve this,
we define
$$
z_{q,r}(r_1) := \left(\frac{nr^2}{\ell^2 r_1'}\right)^{r_1}.
$$
Let us observe that, for $2\le r_1 \le r-1$ we have
\begin{align*}
\frac{z_{q,r}(r_1)}{z_{q,r}(r_1-1)} & = \frac{nr^2}{\ell^2} \left(\frac{r_1^{r_1}}{(r_1-1)^{r_1-1}}\right)^{-1} \\
& = \frac{nr^2}{\ell^2 r_1} \left(1+\frac{1}{r_1-1}\right)^{-(r_1-1)} \\
& \ge \frac{nr}{\ell^2} \cdot e^{-1} \\
& \ge n^{1/4},
\end{align*}
since
$\ell = O\left(\frac{\ln n}{\eps}\right) = o\left(n^{1/3}\ln n\right)$.
Meanwhile we also have
$
\frac{z_{q,r}(1)}{z_{q,r}(0)} = \frac{nr^2}{\ell^2} \ge n^{1/4},
$
 and so
$$
z_{q,r}(r_1) \le z_{q,r}(r-1) n^{-(r-1-r_1)/4}.
$$
Therefore
\begin{align*}
\sum_{r_1=0}^{r-1} z_{q,r}(r_1) & \le z_{q,r}(r-1) \sum_{r_1=0}^{r-1} n^{-(r-1-r_1)/4}\\
& = \left(\frac{nr^2}{\ell^2 \max\{r-1,1\}}\right)^{r-1}(1+o(1)) \\
& \le \left(\frac{2nr}{\ell^2}\right)^{r-1}(1+o(1)),
\end{align*}
which leads to
\begin{align}
x_q(r) & \le y_q(r) \left(\frac{C\ell^2}{n^{2}}\right)^r \left(\frac{2nr}{\ell^2}\right)^{r-1}(1+o(1))\nonumber\\
& = (1+o(1))y_q(r) \frac{\ell^2}{2rn} \left(\frac{2Cr}{n}\right)^r\nonumber\\
& = (1+o(1)) x_q'(r),\label{eq:xratiocase2}
\end{align}
where we define
$$
x_q'(r):=y_q(r) \frac{\ell^2}{2rn} \left(\frac{2Cr}{n}\right)^r.
$$
Now observe that, since the definition of $y_q(r)$ is the same as in Case~1,~\eqref{eq:yratio} still holds, and so
\begin{align*}
\frac{x_q'(r+1)}{x_q'(r)} & = \frac{y_q(r+1)}{y_q(r)} \frac{r}{r+1} \frac{2C(r+1)}{n} \left(\frac{r+1}{r}\right)^r \\
& \le q \cdot \frac{2Cr}{n} \cdot e = O\left(\frac{q^2}{n}\right) \le n^{-1/4},
\end{align*}
since $q \le \ell = o\left(n^{1/3}\ln n\right)$.
We deduce that
\begin{equation}\label{eq:xsumcase2}
\sum_{r=1}^q x_q'(r) \le x_q'(1)\sum_{r=1}^q n^{-(r-1)/4} = (1+o(1))x_q'(1)
\end{equation}
and therefore
$$
\sum_{r=1}^q x_q(r) \stackrel{\eqref{eq:xratiocase2}}{\le} (1+o(1))\sum_{r=1}^q x_q'(r)
\stackrel{\eqref{eq:xsumcase2}}{=} (1+o(1))x_q'(1) = (1+o(1))\frac{C\ell^2}{n^2}. 
$$
Substituting this expression into~\eqref{eqn:squareexpcase2simplified1},
we obtain
\begin{align*}
\sum_{r=1}^{q}\sum_{r_1=0}^{r-1}\sum_{\substack{c_1+\dotsb+c_r=q \\ c_1\geq\dotsb\geq c_r\geq 1}}
\frac{(\ell-q+1)^2\ell^{2(r_2-1)} C^r}{p^q (n-v)^{T(r_1,r_2)}} \left(\frac{r^2}{r_1'}\right)^{r_1}
& =
(1+o(1))\frac{C(\ell-q+1)^2}{(1-\eps)^q n^{k-1}} \\
& = O\left(n^{-j}\right)\frac{(\ell-q+1)^2}{(1-\eps)^q }
\end{align*}
since $j=k-1$.
\end{proof}

Finally observe that Claim~\ref{claim:roneterm} from Case~1 is still
valid for this case. Thus as before we can combine the auxiliary results
to prove the lower bound.

\begin{proof}[Proof of Lemma~\ref{lem:lowerbound-subcritical:case2}]
By substituting the bound from Proposition~\ref{prop:triplesum}
into~\eqref{eqn:squareexpcase2},
we obtain
\begin{align*}
\EE(X_\ell^2) 
& \leq (n)_{v}^2p^{2\ell}\left(1+ O\left(n^{-j}\right) \sum_{q=1}^\ell \frac{(\ell-q+1)^2}{(1-\eps)^q } \right)\\
& \stackrel{\mbox{{\tiny Cl.\ref{claim:roneterm}}}}{=}  (n)_{v}^2p^{2\ell}\left(1+ O\left(n^{-j}\right) \frac{2(1-\eps)^\ell}{\eps^3} \right),
\end{align*}
and exactly the same argument as in Case~1 shows that
$$
 O\left(n^{-j}\right) \frac{2(1-\eps)^\ell}{\eps^3} = o(1),
$$
so
$$
\EE(X_\ell^2) 
\leq\mathbb (n)_{v}^2p^{2\ell}\left(1+ o(1)\right) = (1+o(1))\EE(X_\ell)^2,
$$
as required.
\end{proof}

\end{document}